\documentclass{elsarticle}
\usepackage{amsmath}
\usepackage{amssymb}
\usepackage{amsfonts}
\usepackage{rotating}
\usepackage{graphicx,color}
\usepackage{floatflt,epsfig}
\usepackage{lineno,hyperref}
\usepackage{enumerate}
\usepackage{colortbl}
\usepackage{array,tabularx,tabulary,booktabs}
\usepackage{longtable}
\usepackage{multirow}
\usepackage{wrapfig}
\usepackage{subcaption}

\newcolumntype{^}{>{\currentrowstyle}}

\journal{arXiv}

\newtheorem{theorem}{Theorem}

\begin{document}
\renewcommand{\abstractname}{Abstract}
\renewcommand{\refname}{References}
\renewcommand{\tablename}{Table.}
\renewcommand{\arraystretch}{0.9}
\thispagestyle{empty}
\sloppy

\begin{frontmatter}
\title{The girths of the cubic Pancake graphs}

\author[01,02]{Elena~V.~Konstantinova}
\ead{e\_konsta@math.nsc.ru}

\author[02]{Son En Gun}
\ead{e.son@g.nsu.ru}

\address[01]{Sobolev Institute of Mathematics, Ak. Koptyug av. 4, Novosibirsk, 630090, Russia}
\address[02]{Novosibisk State University, Pirogova str. 2, Novosibirsk, 630090, Russia}

\begin{abstract}
The Pancake graphs $P_n, n\geqslant 2$, are Cayley graphs over the symmetric group $\mathrm{Sym}_n$ generated by prefix-reversals. There are six generating sets of prefix-reversals of cardinality three which give connected Cayley graphs over the symmetric group known as cubic Pancake graphs. In this paper we study the girth of the cubic Pancake graphs. It is proved that considered cubic Pancake graphs have the girths at most twelve.
\end{abstract}

\begin{keyword} Pancake graph; cubic Pancake graph; prefix-reversal; girth

\vspace{\baselineskip}
\MSC[2010] 05C25\sep 05C38\sep 05C82\sep 68R10
\end{keyword}
\end{frontmatter}

\section{Introduction}

{\it The Pancake graph} $P_n=(\mathrm{Sym}_n,PR), n\geqslant 2$, is the Cayley graph over the symmetric group $\mathrm{Sym}_n$ of permutations $\pi=[\pi_1 \pi_2 \ldots \pi_n]$ written as strings in one-line notation, where $\pi_i=\pi(i)$ for any $1\leqslant i \leqslant n$, with the generating set $PR=\{r_i\in \mathrm{Sym}_n: 2 \leqslant i \leqslant n\}$ of all prefix--reversals $r_i$ inversing the order of any substring $[1,i], 2\leqslant i \leqslant n$, of a permutation $\pi$ when multiplied on the right, i.e. $[\pi_1 \ldots \pi_i \pi_{i+1} \ldots \pi_n] r_i=[\pi_i \ldots \pi_1 \pi_{i+1} \ldots \pi_n]$. This graph is well known because of the open combinatorial Pancake problem of finding its diameter~\cite{DW75}.

The graph $P_n$ is a connected vertex--transitive $(n-1)$-regular graph of order $n!$ with no loops and multiple edges. It is almost pancyclic~\cite{KF95,STC06} since it contains cycles of length $\ell$, $6\leqslant \ell \leqslant n!$, but doesn't contain cycles of length $3,4$ or $5$. Since the length of the shortest cycle contained in the graph is six, hence we have $g(P_n)=6$ for any $n\geqslant 3$, where $g(P_n)$ is the girth of $P_n$. The girths of the burnt Pancake graphs over the hyperoctahedral group was considered in~\cite{C11}. The (burnt) Pancake graphs are commonly used in computer science to represent interconnection networks~\cite{BS03,SW16,W13}.

Importance of fixed--degree Pancake graphs, in particular, cubic Pancake graphs as models of networks was shown in~\cite{BS03} by D.~W.~Bass and I.~H.~Sudborough. The authors have considered cubic Pancake graphs as induced subgraphs of the Pancake graph $P_n$. The necessary conditions for a set of three prefix--reversals to generate the symmetric group $\mathrm{Sym}_n$ were found. In particular, it was shown that the cubic Pancake graphs over the symmetric group $\mathrm{Sym}_n$, $n \geqslant 4$, are connected  with the following generating sets:
$$BS_1=\{r_2,r_{n-1},r_n\}; \ BS_2=\{r_{n-2},r_{n-1},r_n\}; \ BS_3=\{r_3,r_{n-2},r_n\},  n \mbox{  is even};$$
$$BS_4=\{r_3,r_{n-1},r_n\}, \mbox{ where } n \mbox{ is odd}; \ BS_5=\{r_{n-3},r_{n-1},r_n\}, n \mbox{ is odd};$$
$$BS_6=\{r_{n-3},r_{n-2},r_n\} \mbox{ for any } n \geqslant 5.$$
The set $BS_2$ is known as 'big-3' flips, and the corresponding cubic Pancake graph generated by this set is called as the Big-3 Pancake network~\cite{SW16}. J.~Sawada and A.~Williams have conjectured in~\cite{SW16} that this graph has cyclic Gray codes.

In this paper we study cubic Pancake graphs and their girths. Our main result is obtained for the cubic Pancake graphs $P_n^i=Cay(\mathrm{Sym}_n,BS_i)$ that are Cayley graphs over the symmetric group $\mathrm{Sym}_n$ generated by the sets $BS_i,\ i=1,\ldots,5$.

\begin{theorem}\label{girth-six-generating-sets}
\begin{equation} \label{BS123}
g(P_n^1, P_n^2, P_n^3)=
\begin{cases}
   6, &\text{when $n=4$;}\\
   8, &\text{when $n \geqslant 5$;}
 \end{cases}
\end{equation}

\begin{equation} \label{BS4}
g(P_n^4)=8, \ \ \  \text{when $n \geqslant 5$ is odd;}
\end{equation}

\begin{equation} \label{BS5}
g(P_n^5)=
\begin{cases}
   8, &\text{when $n=5$;}\\
  12, &\text{when $n \geqslant 7$ is odd;}
 \end{cases}
\end{equation}
\end{theorem}

For $5 \leqslant n \leqslant 19$, the computational results on the girths of the cubic Pancake graph $P_n^6=Cay(\mathrm{Sym}_n,BS_6)$ are given as follows:

\begin{center}
\begin{tabular}{|c|c|c|c|c|c|c|c|c|c|c|c|c|c|c|c|c|}
\hline $n$         & 5 & 6 & 7 & 8 & 9 & 10 & 11 & 12 & 13 & 14 & 15 & 16 & 17 & 18 & 19 \\
\hline $g(P_n^6)$  & 6 & 8 & 10 & 12 & 12 & 16 & 16 & 16 & 20 & 20 & 20 & 24 & 24 & 24 & 28 \\
\hline
\end{tabular}
\end{center}

For any $19 \leqslant n \leqslant 33$, the computational results show that $g(P_n^6)=28$. One can conjecture  that this is true for any $n \geqslant 19$.

The paper is organized as follows. The proof of Theorem~\ref{girth-six-generating-sets} is based on the characterization of small cycles in the Pancake graphs. We present preliminary results with main definitions and notation in Section~\ref{sec1}, where it is shown that for any $n\geqslant 7$ there are no $10$--cycles and $11$--cycles in $P_n^5$.  Then we prove Theorem~\ref{girth-six-generating-sets} in Section~\ref{proof}.

\section{Preliminary results}\label{sec1}

The first results on a characterization of small cycles in the Pancake graph were obtained in~\cite{KM10} where the following cycle representation via a product of generating elements was used. A sequence of prefix--reversals $C_{\ell}=r_{i_0} \ldots r_{i_{\ell-1}}$, where $2\leqslant i_j \leqslant n$, and $i_j \neq i_{j+1}$ for any $0\leqslant j \leqslant \ell-1$, such that $\pi r_{i_0}\ldots r_{i_{\ell-1}}=\pi$, where $\pi \in \mathrm{Sym}_n,$ is called {\it a form of a cycle $C_{\ell}$ of length $\ell$}. A cycle $C_{\ell}$ of length $\ell$ is also called an $\ell$--cycle, and a vertex of $P_n$ is identified with the permutation which corresponds to this vertex. It is evident that any $\ell$--cycle can be presented by $2\,\ell$ forms (not necessarily different) with respect to a vertex and a direction. {\it The canonical form $C_{\ell}$ of an $\ell$--cycle} is called a form with a lexicographically maximal sequence of indices $i_0\ldots i_{\ell-1}$.  We shortly write $C_{\ell}=(r_a r_b)^k$ for a cycle form $C_{\ell}=r_a r_b\ldots r_a r_b$, where $\ell=2\,k, \ a\neq b$, $r_a r_b$ appears exactly $k$ times and $\pi\,r_a r_b\ldots r_a r_b=\pi$ for any $\pi \in \mathrm{Sym}_n$.  The form $C_{\ell}=(r_a r_b)^k$ is \textit{canonical} if $a>b$. By using this description, the following results characterizing $6$-- and $7$--cycles were obtained.

\begin{theorem} {\rm \cite{KM10}} \label{67-cycles}
The Pancake graph $P_n, n\geqslant 3,$ has $\frac{n!}{6}$ independent $6$--cycles of the canonical form
\begin{equation} \label{C6}
C_6=(r_3\,r_2)^3
\end{equation}
and $n! (n-3)$ distinct $7$--cycles of the canonical form
\begin{equation} \label{C7}
C_7=r_k\,r_{k-1}\,r_k\,r_{k-1}\,r_{k-2}\,r_k\,r_2,
\end{equation}
 where $4\leqslant k \leqslant n$. Each of the vertices of $P_n$ belongs to exactly one $6$--cycle and $7 (n-3)$ distinct $7$--cycles.
\end{theorem}

The complete characterization of $8$--cycles is given by the following theorem.

\begin{theorem} {\rm \cite{KM14}} \label{8-cycles} Each of vertices of $P_n, n\geqslant 4,$ belongs to $N=\frac{n^3+12n^2-103n+176}{2}$ distinct $8$--cycles of the following canonical forms:
\begin{equation} \label{C81}
 C^1_8=r_k\,r_j\,r_i\,r_j\,r_k\,r_{k-j+i}\,r_i\,r_{k-j+i},\hspace{10mm}  2\leqslant i < j \leqslant k-1,\,\,4\leqslant k \leqslant n
\end{equation}
\begin{equation} \label{C82}
C^2_8=r_k\,r_{k-1}\,r_2\,r_{k-1}\,r_k\,r_2\,r_3\,r_2, \hspace{46mm} 4\leqslant k \leqslant n
\end{equation}
\begin{equation} \label{C83}
C^3_8=r_k\,r_{k-i}\,r_{k-1}\,r_i\,r_k\,r_{k-i}\,r_{k-1}\,r_i,\hspace{16mm} 2\leqslant i \leqslant k-2,\,\,4\leqslant k \leqslant n
\end{equation}
\begin{equation} \label{C84}
C^4_8=r_k\,r_{k-i+1}\,r_k\,r_i\,r_k\,r_{k-i}\,r_{k-1}\,r_{i-1},\hspace{12mm} 3\leqslant i\leqslant k-2,\,\,5\leqslant k \leqslant n
\end{equation}
\begin{equation} \label{C85}
C^5_8=r_k\,r_{k-1}\,r_{i-1}\,r_k\,r_{k-i+1}\,r_{k-i}\,r_k\,r_i,\hspace{10mm} 3\leqslant i\leqslant k-2,\,\,5\leqslant k \leqslant n
\end{equation}
\begin{equation} \label{C86}
C^6_8=r_k\,r_{k-1}\,r_k\,r_{k-i}\,r_{k-i-1}\,r_k\,r_i\,r_{i+1},\hspace{10mm} 2\leqslant i\leqslant k-3,\,\,5\leqslant k \leqslant n
\end{equation}
\begin{equation}\label{C87}
C_8^7=r_k\,r_{k-j+1}\,r_k\,r_i\,r_k\,r_{k-j+1}\,r_k\,r_i,\hspace{7mm} 2\leqslant i < j\leqslant k-1,\,\,4\leqslant k \leqslant n
\end{equation}
\begin{equation}\label{C88}
C_8^8=(r_4\,r_3)^4.\hspace{84mm}
\end{equation}
\end{theorem}

The complete characterization of $9$--cycles in the Pancake graphs were obtained in~\cite{KM11}.

In general, the complete characterization of small cycles in the Pancake graphs presented in~\cite{KM10,KM11} is based on the hierarchical structure of the Pancake graphs. The graph \mbox{$P_n, n\geqslant 4,$} is constructed from $n$ copies of $P_{n-1}(i), 1\leqslant i \leqslant n,$ such that each $P_{n-1}(i)$ has the vertex set $$V_i=\lbrace[\pi_1\dots\pi_{n-1}i],$$ where $\pi_k\in\lbrace 1,\dots,n\rbrace \setminus \lbrace i \rbrace : 1\leqslant k \leqslant n-1\rbrace,$ $|V_i| =(n-1)!,$ and the edge set $$E_i=\lbrace\lbrace[\pi_1\dots\pi_{n-1}i],[\pi_1\dots\pi_{n-1}i]r_j\rbrace:2\leqslant j\leqslant n-1\rbrace,$$ where $|E_i|=\frac{(n-1)!(n-2)}{2}.$ Any two copies $P_{n-1}(i)$ and $P_{n-1}(j), i \neq j,$ are connected by $(n-2)!$ edges presented as $\lbrace[i\pi_2\dots\pi_{n-1}j],[j\pi_(n-1\dots\pi_{2}i]\rbrace,$ where $[i\pi_2\dots\pi_{n-1}j]r_n=[j\pi_(n-1\dots\pi_{2}i].$ Prefix-reversals $r_j,$ $2\leqslant j \leqslant n-1,$ defines \textit{internal edges} in $P_{n-1}(i), 1\leqslant i \leqslant n$, and the prefix-reversal $r_n$ defines \textit{external edges} between copies. Copies $P_{n-1}(i)$, or just $P_{n-1}$ when it is not important to specify the last element of permutations belonging to the copy, are also called $(n-1)-$copies.

The hierarchical structure of the Pancake graphs is used to prove the following two results.

\begin{theorem} \label{10-cycles} In the cubic Pancake graphs $P_n^5=Cay(\mathrm{Sym}_n,\{r_{n-3},r_{n-1},r_n\})$, $n\geqslant 7,$ there are no cycles of length $10$. For $n=5$ there are $10$-cycles of the canonical form $C_{10}=(r_5\,r_4)^5$.
\end{theorem}

\begin{proof} Since the cubic Pancake graphs $P_n^{5}, n\geqslant 5,$ are induced subgraphs of the Pancake graph $P_n$, then let us consider all possible cases for forming $10$--cycles in the Pancake graphs with taking into account that the generating set of $P_n^{5}$ contains only three elements $r_{n-3},r_{n-1},r_n$, where $n$ is odd. If $n=5$ then there are cycles of length $10$ of the canonical form $C_{10}=(r_5\,r_4)^5$ (see~Lemma~1 in~\cite{KM16}). If $n\geqslant 7$ then due to the hierarchical structure of the Pancake graph $P_n$, cycles of length $10$ could be formed from paths of length $l, \ 2\leqslant l\leqslant 8$, belonging to different $(n-1)-$copies of $P_n$.

Further, we consider all possible options for the distribution of vertices by copies. Without loss of generality we always put $\tau^1=I_n=[1\,2\,3\,\dots\,n-1\,n]$.

\vspace{5mm}
{\bf Case 1: 10-cycle within $P_n$ has vertices from two copies of $P_{n-1}$}
\vspace{5mm}

Suppose that a sought $10$--cycle is formed on  vertices from copies $P_{n-1}(i)$ and $P_{n-1}(j)$, where $1\leqslant i\neq j\leqslant n$. It was shown in~\cite[Lemma 2]{KM10} that if two vertices $\pi$ and $\tau,$ belonging to the same $(n-1)-$copy, are at the distance at most two, then their external neighbours $\overline{\pi}$ and $\overline{\tau}$ should belong to distinct $(n-1)-$copies. Hence, a sought cycle cannot occur in situations when its two (three) vertices belong to one copy and eight (seven) vertices belong to another one. Therefore, such a cycle must have at least four vertices in each of the two copies. Hence, there are two following cases.

\vspace{3mm}

{\bf \underline{Case $(4+6).$}} Suppose that four vertices $\pi^1,\pi^2,\pi^3,\pi^4$ of a sought $10$--cycle belong to one copy, and other six vertices $\tau^1,\tau^2,\tau^3,\tau^4,\tau^5,\tau^6$ belong to another copy. Let $\pi^1=\tau^1 r_n$ and $\pi^4=\tau^6 r_n$. Since $\tau^1=I_n$ then $\pi^1$ and $\pi^4$ should belong to $P_{n-1}(1)$. Herewith, the four vertices of $P_{n-1}(1)$ should form a path of length three whose endpoints should be adjacent to vertices from $P_{n-1}(n)$.

Consider all options for passing from $\tau^1$ to $\tau^6$ by internal edges in a copy $P_{n-1}(n)$. Since the generating set of $P_n^5$ consists of the elements $r_{n-1}$ and $r_{n-3}$ corresponding to internal edges then there are two ways to get paths of length five from $\tau^1$ to $\tau^6$ (see Figure~1).

The first path is presented as follows:
\begin{equation}\label{tau1_to_tau6_1}
\tau^1\,(r_{n-3}\,r_{n-1})^2\,r_{n-3}=\tau^6
\end{equation}
such that:
$$\tau^1=[1\,2\,3\,\dots\,n-1\,n]\xrightarrow{r_{n-3}}[n-3\,n-4\,\dots\,2\,1\,n-2\,n-1\,n]\xrightarrow{r_{n-1}}$$ $$[n-1\,n-2\,1\,2\,\dots\,n-4\,n-3\,n]  \xrightarrow{r_{n-3}}[n-5\,n-6\,\dots\,2\,1\,n-2\,n-1\,n-4\,n-3\,n]\xrightarrow{r_{n-1}}$$ $$[n-3\,n-4\,n-1\,n-2\,1\,2\,\dots\,n-6\,n-5\,n]\xrightarrow{r_{n-3}}$$
$$[n-7\,n-8\,\dots\,2\,1\,n-2\,n-1\,n-4\,n-3\,n-6\,n-5\,n]=\tau^6.$$

Since $\pi^4=\tau^6\,r_n$ then we have:

\begin{equation}\label{pi4_1}
\pi^4=[n\,n-5\,n-6\,n-3\,n-4\,n-1\,n-2\,1\,2\,\dots\,n-8\,n-7].
\end{equation}

\vspace{5mm}

\begin{figure}[ht]\centering
\includegraphics[scale=0.6]{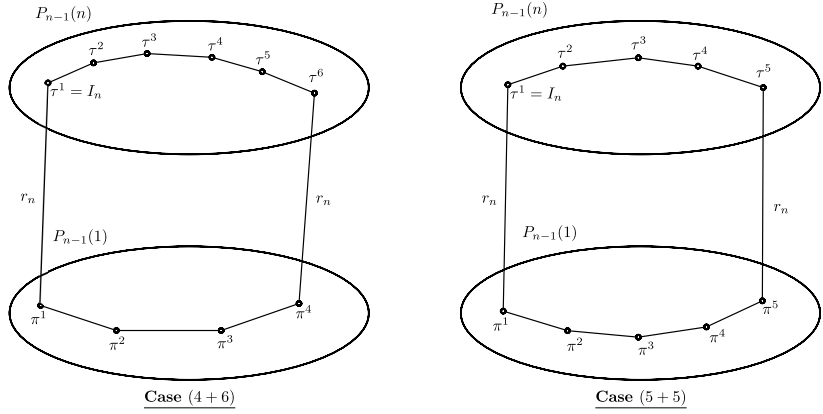}
\caption{Case 1: 10-cycle within $P_n$ has vertices from two copies of $P_{n-1}$}\label{2copy}
\end{figure}

Note that $\pi^1=\tau^1\,r_n=[n\,n-1\,\dots\,2\,1]$ with $\pi_n^1=1$ for any $n\geqslant 5$. Hence, we immediately can conclude that $\pi^4$ given by~(\ref{pi4_1}) and $\pi^1$ belong to different copies of $P_n$ since $\pi_n^4\neq 1$ for any odd $n\geqslant 7$. For $n=5$ we get $\pi^4=[5\,3\,4\,2\,1]$ and there is no path of length three between $\pi^4$ and $\pi^1$, presented as $r_{n-3}\,r_{n-1}\,r_{n-3}$ or $r_{n-1}\,r_{n-3}\,r_{n-1}$, since $\pi^4\,r_2\,r_4\,r_2=[4\,2\,5\,3\,1]$ and $\pi^4\,r_4\,r_2\,r_4=[5\,3\,2\,4\,1]$. This gives a contradiction with an assumption that $\pi^1$ and $\pi^4$ belong to the same copy $P_{n-1}(1)$. Thus, a sought cycle cannot occur neither in $P_n$ nor in $P_n^5$.

The second path is presented as follows:
\begin{equation}\label{tau1_to_tau6_2}
\tau^1\,(r_{n-1}\,r_{n-3})^2\,r_{n-1}=\tau^6
\end{equation}
such that
$$\tau^1=[1\,2\,3\,\dots\,n-1\,n]\xrightarrow{r_{n-1}}[n-1\,n-2\,\dots\,2\,1\,n]\xrightarrow{r_{n-3}}[3\,4\,\dots\,n-1\,2\,1\,n]  \xrightarrow{r_{n-1}}$$
$$[1\,2\,n-1\,n-2\,\dots\,4\,3\,n]\xrightarrow{r_{n-3}}[5\,6\,\dots\,n-2\,n-1\,2\,1\,4\,3\,n]\xrightarrow{r_{n-1}}$$
$$[3\,4\,1\,2\,n-1\,n-2\,\dots\,6\,5\,n]=\tau^6,$$
and hence
$$\pi^4=\tau^6\,r_n=[n\,5\,6\,\dots\,n-2\,n-1\,2\,1\,4\,3],$$
which means that this permutation belongs to $P_{n-1}(3)$. However, $\pi^1$ belongs to $P_{n-1}(1)$ which gives a contradiction again. Thus, a sought $10$--cycle cannot occur in this case.

{\bf \underline{Case $(5+5).$}} Suppose that five vertices $\pi^1,\pi^2,\pi^3,\pi^4, \pi^5$ of a sought $10$--cycle belong to a copy $P_{n-1}(1)$, and other five vertices $\tau^1,\tau^2,\tau^3,\tau^4,\tau^5$ belong to a copy $P_{n-1}(n)$, where $\tau^1=I_n$, $\pi^1=\tau^1 r_n$, and $\pi^5=\tau^5 r_n$. Then the five vertices of $P_{n-1}(1)$ should form a path of length four whose endpoints should be adjacent to the vertices from $P_{n-1}(n)$ (see Figure~1).

There are two ways to get paths of length four from $\tau^1$ to $\tau^5$.

The first way is given as follows:
$$\tau^1\,(r_{n-3}\,r_{n-1})^2=\tau^5,$$
more precisely, we have: $$\tau^1=[1\,2\,3\,\dots\,n-1\,n]\xrightarrow{r_{n-3}}[n-3\,n-4\,\dots\,2\,1\,n-2\,n-1\,n]\xrightarrow{r_{n-1}}$$
$$[n-1\,n-2\,1\,2\,\dots\,n-4\,n-3\,n]\xrightarrow{r_{n-3}}[n-5\,n-6\,\dots\,2\,1\,n-2\,n-1\,n-4\,n-3\,n]\xrightarrow{r_{n-1}}$$
$$[n-3\,n-4\,n-1\,n-2\,1\,2\,\dots\,n-6\,n-5\,n]=\tau^5.$$

The second way is presented as follows:
$$\tau^1\,(r_{n-1}\,r_{n-3})^2=\tau^5,$$
and we have:
$$\tau^1=[1\,2\,3\,\dots\,n-1\,n]\xrightarrow{r_{n-1}}[n-1\,n-2\,\dots\,2\,1\,n]\xrightarrow{r_{n-3}}[3\,4\,\dots\,n-1\,2\,1\,n]\xrightarrow{r_{n-1}}$$
$$[1\,2\,n-1\,n-2\,\dots\,4\,3\,n]\xrightarrow{r_{n-3}}[5\,6\,\dots\,n-2\,n-1\,2\,1\,4\,3\,n]=\tau^5.$$

Since $\pi^5=\tau^5\,r_n$ then we have either:
$$\pi^5=[n\,n-5\,n-6\,\dots\,2\,1\,n-2\,n-1\,n-4\,n-3]$$
or
$$\pi^5=[n\,3\,4\,\,1\,2\,n-1\,n-2\,\dots\,6\,5],$$
which gives a contradiction with an assumption that $\pi^1$ and $\pi^5$ belong to the copy $P_{n-1}(1)$ for any odd $n\geqslant 5$. Thus, in this case a sought cycle cannot occur in $P_n$, and hence in $P_n^5$.

\vspace{5mm}
{\bf Case 2: 10-cycle within $P_n$ has vertices from three copies of $P_{n-1}$}
\vspace{5mm}

There are four possible situations in this case.

{\bf \underline{Case $(2+2+6).$}} Suppose that two vertices $\pi^1,\pi^2$ of a sought $10$--cycle belong to one copy, other two vertices $\tau^1,\tau^2$ belong to another copy, and remaining six vertices $\gamma^1,\gamma^2, \gamma^3, \gamma^4, \gamma^5, \gamma^6$ belong to the third copy. Let $\pi^1=\tau^2 r_n$, $\pi^2=\gamma^6 r_n$, and $\tau^1=\gamma^1 r_n=I_n$, then $\gamma^1$ and $\gamma^6$ should belong to $P_{n-1}(1)$ (see Figure~2).

It is evident that there are two ways to get paths of length one from $\tau^1$ to $\tau^2$:
$$\tau^1=[1\,2\,3\,\dots\,n-1\,n]\xrightarrow{r_{n-3}}[n-3\,n-4\,\dots\,2\,1\,n-2\,n-1\,n]=\tau^2$$
or
$$\tau^1=[1\,2\,3\,\dots\,n-1\,n]\xrightarrow{r_{n-1}}[n-1\,n-2\,\dots\,2\,1\,n]=\tau^2.$$

Since $\pi^1=\tau^2\,r_n$, we have either
$$\pi^1=[n\,n-1\,n-2\,1\,2\,\dots\,n-4\,n-3]$$
or
$$\pi^1=[n\,1\,2\,\dots\,n-2\,n-1].$$

Similar, there are two ways to get a path of length one from $\pi^1$ to $\pi^2$ such that the first way is given by $\pi^1 r_{n-3}=\pi^2$, where we have either
$$\pi^2=[n-6\,n-7\,\dots\,2\,1\,n-2\,n-1\,n\,n-5\,n-4\,n-3]$$
or
$$\pi^2=[n-4\,n-5\,\dots\,2\,1\,n\,n-3\,n-2\,n-1],$$
and the second way is given by $\pi^1 r_{n-1}=\pi^2$, where we have either
$$\pi^2=[n-4\,n-5\,\dots\,2\,1\,n-2\,n-1\,n\,n-3]$$
or
$$\pi^2=[n-2\,n-3\,\dots\,2\,1\,n\,n-1],$$

\begin{figure}[ht]\centering
\includegraphics[scale=0.6]{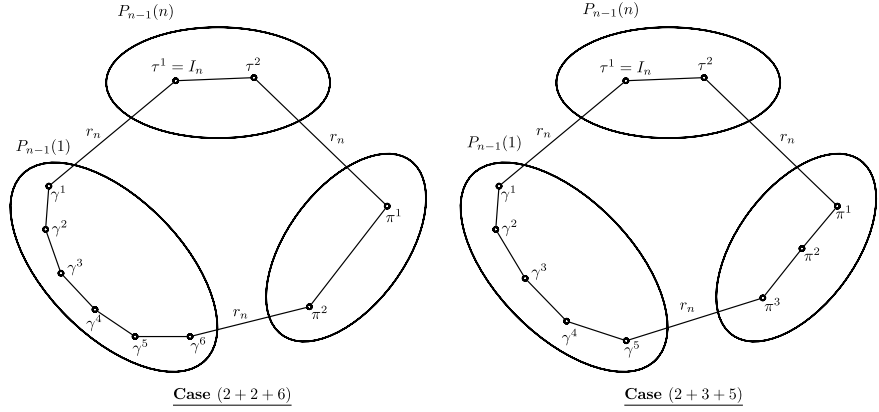}
\caption{Case 2: 10-cycle within $P_n$ has vertices from three copies of $P_{n-1}$}\label{3copy(226+235)}
\end{figure}

and since $\gamma^6=\pi^2\,r_n$ we get:
$$\gamma^6 =
  \begin{cases}
     [n-3\,n-4\,n-5\,n\,n-1\,n-2\,1\,2\,\dots\,n-7\,n-6]=\gamma^6(A)&or\\
     [n-1\,n-2\,n-3\,n\,1\,2\,\ldots\,n-5\,n-4]=\gamma^6(B)&or\\
     [n-3\,n\,n-1\,n-2\,1\,2\,\dots\,n-5\,n-4]=\gamma^6(C)&or\\
     [n-1\,n\,1\,2\,\dots\,n-3\,n-2]=\gamma^6(D).\\
  \end{cases}$$

To get a sought $10$-cycle there should be a path of length five between $\gamma^6$ and $\gamma^1$, where $\gamma^1=\tau^1\,r_n=[n\,n-1\,\dots\,2\,1]$. Let us check this. If $n=5$ the vertices $\gamma^6(B)$, $\gamma^6(C)$ and $\gamma^1=[5\,4\,3\,2\,1]$ belong to the copy $P_{n-1}(1)$, and there are two ways to get a path of length five from $\gamma^6$ and $\gamma^1$. Namely, applying $(r_2\,r_4)^2r_2$ to $\gamma^6$ we have either:
$$\gamma^6(B)=[4\,3\,2\,5\,1]\xrightarrow{r_2}[3\,4\,2\,5\,1]\xrightarrow{r_4}[5\,2\,4\,3\,1]\xrightarrow{r_2}[2\,5\,4\,3\,1]\xrightarrow{r_4}$$
$$[3\,4\,5\,2\,1]\xrightarrow{r_2}[4\,3\,5\,2\,1]\neq \gamma^1$$
or
$$\gamma^6(C)=[2\,5\,4\,3\,1]\xrightarrow{r_2}[5\,2\,4\,3\,1]\xrightarrow{r_4}[3\,4\,2\,5\,1]\xrightarrow{r_2}[4\,3\,2\,5\,1]\xrightarrow{r_4}$$
$$[5\,2\,3\,4\,1]\xrightarrow{r_2}[2\,5\,3\,4\,1]\neq \gamma^1,$$
and applying $(r_4r_2)^2r_{4}$ to $\gamma^6$ we have either:
$$\gamma^6(B)=[4\,3\,2\,5\,1]\xrightarrow{r_4}[5\,2\,3\,4\,1]\xrightarrow{r_2}[2\,5\,3\,4\,1]\xrightarrow{r_4}[4\,3\,5\,2\,1]\xrightarrow{r_2}$$
$$[3\,4\,5\,2\,1]\xrightarrow{r_4}[2\,5\,4\,3\,1]\neq \gamma^1$$
or
$$\gamma^6(C)=[2\,5\,4\,3\,1]\xrightarrow{r_4}[3\,4\,5\,2\,1]\xrightarrow{r_2}[4\,3\,5\,2\,1]\xrightarrow{r_4}[2\,5\,3\,4\,1]\xrightarrow{r_2}$$
$$[5\,2\,3\,4\,1]\xrightarrow{r_4}[4\,3\,2\,5\,1]\neq \gamma^1.$$
Hence, a path of length five does not occur between $\gamma^6$ and $\gamma^1$.

If $n=7$ the vertices $\gamma^6(A)$ and $\gamma^1=[7\,6\,5\,4\,3\,2\,1]$ belong to the copy $P_{n-1}(1)$, and the following two cases are possible:
$$\gamma^6(A)(r_{4}\,r_{6})^2r_{4}=[4\,3\,2\,7\,6\,5\,1](r_4\,r_6)^2r_4=[6\,5\,2\,7\,4\,3\,1]\neq \gamma^1$$
or
$$\gamma^6(A)(r_{6}\,r_{4})^2r_{6}=[4\,3\,2\,7\,6\,5\,1](r_6\,r_4)^2r_6=[2\,7\,4\,3\,5\,6\,1]\neq \gamma^1.$$
Again, a path of length five does not occur between $\gamma^6$ and $\gamma^1$.

If $n\geqslant 9$, there is a contradiction with an assumption that $\gamma^1$ and $\gamma^6$ belong to the copy $P_{n-1}(1)$. Thus, a sought cycle cannot occur in this case.

{\bf \underline{Case $(2+3+5).$}} Suppose that two vertices $\tau^1,\tau^2$ of a sought $10$--cycle belong to one copy, other three vertices $\pi^1,\pi^2,\pi^3$ belong to another copy, and remaining five vertices $\gamma^1,\gamma^2, \gamma^3, \gamma^4, \gamma^5$ belong to the third copy. Let $\pi^1=\tau^2 r_n$, $\pi^3=\gamma^5 r_n$, and $\tau^1=\gamma^1 r_n=I_n$, then $\gamma^1$ and $\gamma^5$ should belong to $P_{n-1}(1)$
(see Figure~2). There are two ways to get a path of length two from $\tau^1=I_n$ to $\pi^1$ such that either
$$\tau^1 r_{n-3} r_n =[n\,n-1\,n-2\,1\,2\,\dots\,n-4\,n-3] = \pi^1$$
or
$$\tau^1 r_{n-1} r_n=[n\,1\,2\,\dots\,n-2\,n-1] = \pi^1.$$

Similar, there are two ways to get a path of length two from $\pi^1$ to $\pi^3$. The first one is presented as follows:
\begin{equation}\label{pi1_to_pi3_1}
\pi^1 r_{n-3}r_{n-1}=\pi^3,
\end{equation}
such that either
$$\pi^1=[n\,n-1\,n-2\,1\,2\,\dots\,n-7\,n-6\,n-5\,n-4\,n-3]\xrightarrow{r_{n-3}}$$ \vspace{-5mm}
$$[n-6\,n-7\,\dots\,2\,1\,n-2\,n-1\,n\,n-5\,n-4\,n-3]\xrightarrow{r_{n-1}}$$
$$[n-4\,n-5\,n\,n-1\,n-2\,1\,2\,\dots\,n-7\,n-6\,n-3]=\pi^3$$
or
$$\pi^1=[n\,1\,2\,\dots\,n-2\,n-1]\xrightarrow{r_{n-3}}[n-4\,n-5\,\dots\,2\,1\,n\,n-3\,n-2\,n-1]\xrightarrow{r_{n-1}}$$
$$[n-2\,n-3\,n\,1\,2\,3\,\dots\,n-5\,n-4\,n-1]=\pi^3.$$
The second way is presented as follows:
\begin{equation}\label{pi1_to_pi3_2}
\pi^1 r_{n-1}r_{n-3}=\pi^3,
\end{equation}
such that either
$$\pi^1=[n\,n-1\,n-2\,1\,2\,\dots\,n-4\,n-3]\xrightarrow{r_{n-1}}[n-4\,n-5\,\dots\,2\,1\,n-2\,n-1\,n\,n-3]\xrightarrow{r_{n-3}}$$
$$[n-2\,1\,2\,3\,\dots\,n-4\,n-1\,n\,n-3]=\pi^3$$
or
$$\pi^1=[n\,1\,2\,\dots\,n-2\,n-1]\xrightarrow{r_{n-1}}[n-2\,n-3\,\dots\,2\,1\,n\,n-1]\xrightarrow{r_{n-3}}$$ $$[2\,3\,\dots\,n-2\,1\,n\,n-1]=\pi^3.$$

Since $\pi^3=\gamma^5\,r_n$, then by~(\ref{pi1_to_pi3_1}) and~(\ref{pi1_to_pi3_2}) we obtain:
$$\gamma^5 =
  \begin{cases}
     [n-3\,n-6\,n-7\,\dots2\,1\,n-2\,n-1\,n\,n-5\,n-4]=\gamma^5(A)&or\\
     [n-1\,n-4\,\dots\,3\,2\,1\,n\,n-3\,n-2]=\gamma^5(B)&or\\
     [n-3\,n\,n-1\,n-4\,\dots\,3\,2\,1\,n-2]=\gamma^5(C)&or\\
     [n-1\,n\,\,1\,n-2\,\dots\,3\,2]=\gamma^5(D).\\
  \end{cases}$$

To get a sought $10$-cycle there should be a path of length four between $\gamma^5$ and $\gamma^1$, where $\gamma^1=\tau^1\,r_n=[n\,n-1\,\dots\,2\,1]$. Let us check this. If $n=5$ the vertices $\gamma^5(A)=I_n\,r_{n-3}\,r_n\,r_{n-3}\,r_{n-1}\,r_n=[2\,4\,5\,3\,1]$ and $\gamma^1=[5\,4\,3\,2\,1]$ belong to the copy $P_{n-1}(1)$, and the following cases are possible:
$$\gamma^5(A)(r_{n-3}\,r_{n-1})^2=[2\,4\,5\,3\,1](r_2\,r_4)^2=[4\,2\,3\,5\,1]\neq \gamma^1$$
or
$$\gamma^5(A)(r_{n-1}\,r_{n-3})^2=[2\,4\,5\,3\,1](r_4\,r_2)^2=[4\,2\,3\,5\,1]\neq \gamma^1.$$
Hence, a path of length four does not occur between $\gamma^5$ and $\gamma^1$. However, let us note that in this case we have a cycle of length eight given by the canonical form $C_8=(r_4\,r_2)^4$ obtained from~(\ref{C87}) by putting $k=4,\,j=3,\,i=2$.

If $n\geqslant 7$, there is a contradiction with an assumption that $\gamma^1$ and $\gamma^6$ belong to the copy $P_{n-1}(1)$. Thus, a sought cycle cannot occur in this case.

{\bf \underline{Case $(2+4+4).$}} Suppose that two vertices $\tau^1,\tau^2$ of a sought $10$--cycle belong to one copy, other four vertices $\pi^1,\pi^2,\pi^3,\pi^4$ belong to another copy, and remaining four vertices $\gamma^1,\gamma^2, \gamma^3, \gamma^4$ belong to the third copy (see Figure~3). Let $\pi^1=\tau^1 r_n=[n\,n-1\,n-2\,\dots\,2\,1]$, $\pi^4=\gamma^4 r_n$, and $\tau^2=\gamma^1 r_n$. Then both $\gamma^1$ and $\gamma^4$ should belong to either $P_{n-1}(n-1)$ or $P_{n-1}(n-3)$, since either $\gamma^1=\tau^1 r_{n-1}r_n$ or $\gamma^1=\tau^1 r_{n-3}r_n$. More precisely, we have:
\begin{equation}\label{gamma1AB}
\gamma^1=
  \begin{cases}
     [n\,1\,2\,\dots\,n-2\,n-1]=\gamma^1(A)& or\\
     [n\,n-1\,n-2\,1\,2\,\dots\,n-4\,n-3]=\gamma^1(B).\\
  \end{cases}
\end{equation}

On the other hand, there are two ways to get a path of length three from $\pi^1$ to $\pi^4$. The first way is presented as follows:
\begin{equation}\label{pi1_to_pi4_1}
\pi^1 r_{n-3}r_{n-1}r_{n-3}=\pi^4
\end{equation}
such that we have:
$$\pi^1=[n\,n-1\,n-2\,\dots\,5\,4\,3\,2\,1]\xrightarrow{r_{n-3}}[4\,5\,\dots\,n-1\,n\,3\,2\,1]\xrightarrow{r_{n-1}}$$
$$[2\,3\,n\,n-1\dots\,5\,4\,1]\xrightarrow{r_{n-3}}[6\,7\,\dots\,n-1\,n\,3\,2\,5\,4\,1]=\pi^4.$$

\begin{figure}[ht]\centering
\includegraphics[scale=0.6]{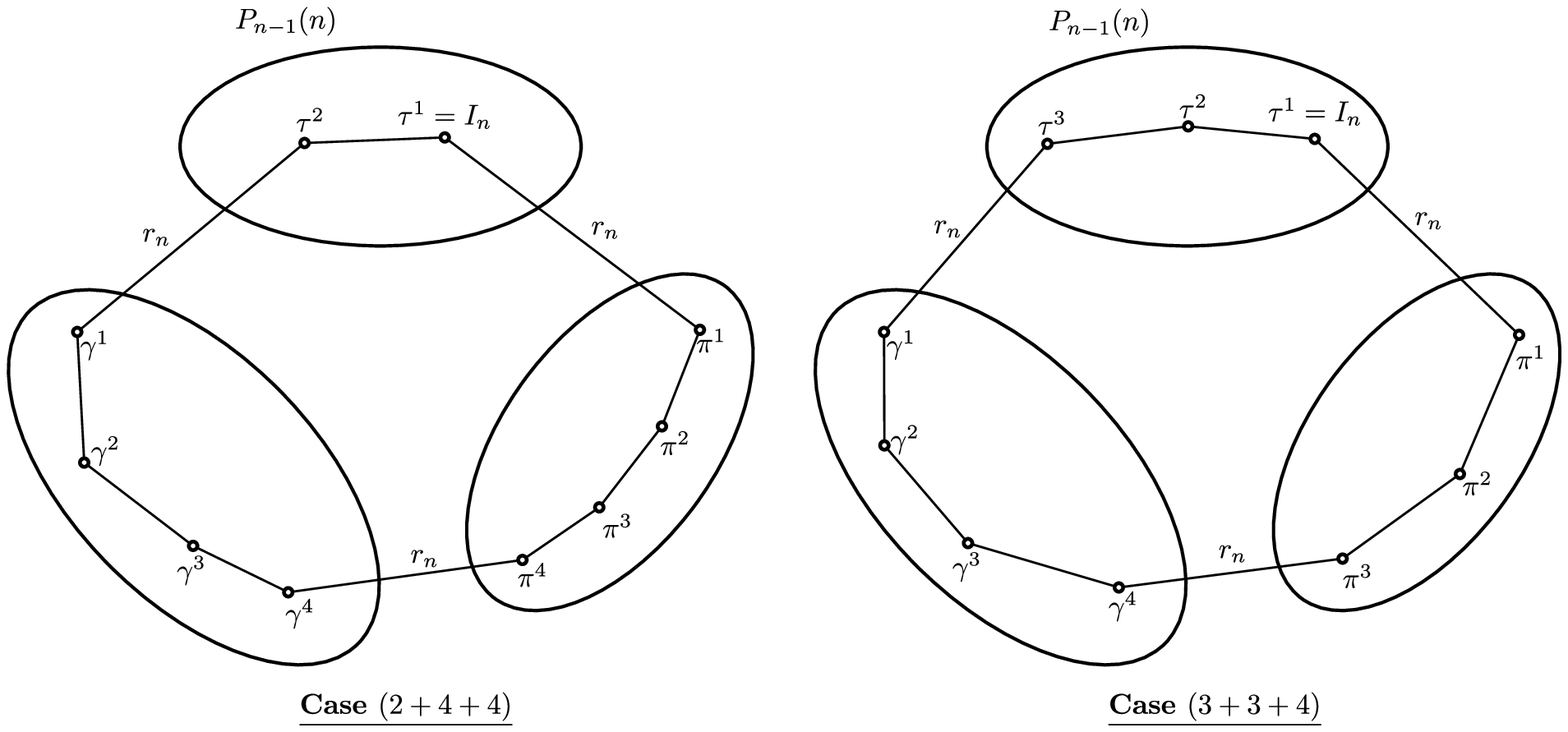}
\caption{Case 2: 10-cycle within $P_n$ has vertices from three copies of $P_{n-1}$}\label{3copy(244+334)}
\end{figure}

The second way is presented as follows:
\begin{equation}\label{pi1_to_pi4_2}
\pi^1 r_{n-1}r_{n-3}r_{n-1}=\pi^4,
\end{equation}
where we have:
$$\pi^1=[n\,n-1\,n-2\,\dots\,2\,1]\xrightarrow{r_{n-1}}[2\,3\,\dots\,n-2\,n-1\,n\,1]\xrightarrow{r_{n-3}}$$
$$[n-2\,n-3\dots\,3\,2\,n-1\,n\,1]\xrightarrow{r_{n-1}}[n\,n-1\,2\,3\,\dots\,n-3\,n-2\,1]=\pi^4.$$

Since $\pi^4=\gamma^4\,r_n,$ then by~(\ref{pi1_to_pi4_1}) and~(\ref{pi1_to_pi4_2}) we obtain:
\begin{equation}\label{gamma1CD}
\gamma^4 =
  \begin{cases}
     [1\,4\,5\,2\,3\,n\,n-1\dots\,7\,6]=\gamma^4(C)& or\\
     [1\,n-2\,\dots\,3\,2\,n-1\,n]=\gamma^4(D).\\
  \end{cases}
\end{equation}

To get a sought $10$-cycle there should be a path of length three between $\gamma^1$ and $\gamma^4$ (see Figure~3, Subcase $(2+4+4)$). Let us check this.

If $n=5$ or $n\geqslant 11$, there is a contradiction with an assumption that both $\gamma^1$ and $\gamma^4$ belong to either $P_{n-1}(n-1)$ or $P_{n-1}(n-3)$.

If $n=7$ then by~(\ref{gamma1AB}) and~(\ref{gamma1CD}) we have $\gamma^1(B)=[7\,1\,2\,3\,4\,5\,6]$ and $\gamma^4(C)=[1\,4\,5\,2\,3\,7\,6]$, but there is no a path of length three between them, since we have either
$$\gamma^4(C)\,r_4\,r_6\,r_4=[4\,1\,3\,7\,5\,2\,6]\neq\gamma^1(B)$$
or
$$\gamma^4(C)\,r_6\,r_4\,r_6=[1\,4\,7\,3\,2\,5\,6]\neq \gamma^1(B).$$

If $n=9$ then by~(\ref{gamma1AB}) and~(\ref{gamma1CD}) we have $\gamma^1(A)=[9\,8\,7\,1\,2\,3\,4\,5\,6]$ and $\gamma^4(C)=[1\,4\,5\,2\,3\,9\,8\,7\,6]$, but there is no a path of length three between them, since we have either
$$\gamma^4(C)\,r_6\,r_8\,r_6=[2\,5\,4\,1\,8\,7\,3\,9\,6]\neq \gamma^1(A)$$
or
$$\gamma^4(C)\,r_8\,r_6\,r_8=[1\,4\,7\,8\,9\,3\,2\,5\,6]\neq \gamma^1(A).$$
Thus, a sought cycle cannot occur in this case.

{\bf \underline{Case $(3+3+4).$}} Suppose that three vertices $\tau^1,\tau^2,\tau^3$ of a sought $10$--cycle belong to one copy, other three vertices $\pi^1,\pi^2,\pi^3$ belong to another copy, and remaining four vertices $\gamma^1,\gamma^2, \gamma^3, \gamma^4$ belong to the third copy (see Figure~3). Let $\pi^1=\tau^1\,r_n$, $\pi^3=\gamma^4\,r_n$, and $\tau^3=\gamma^1\,r_n$. Then both $\gamma^1$ and $\gamma^4$ should belong to either $P_{n-1}(n-1)$ or $P_{n-1}(3)$, since either $\gamma^1=\tau^1 r_{n-1}r_{n-3}r_n$ or $\gamma^1=\tau^1 r_{n-3}r_{n-1}r_n$. More precisely, we have:
\begin{equation}\label{gamma1AB+}
\gamma^1=
\begin{cases}
     [n\,1\,2\,n-1\,n-2\,\dots\,4\,3]=\gamma^1(A)&or\\
     [n\,n-3\,n-4\,\dots\,2\,1\,n-2\,n-1]=\gamma^1(B).\\
  \end{cases}
\end{equation}

On the other hand, since $\pi^1=\tau^1\,r_n=[n\,n-1\,n-2\,\dots\,2\,1]$, then by~(\ref{pi1_to_pi3_1}) and~(\ref{pi1_to_pi3_2}) there are two ways to get a path of length two from $\pi^1$ to $\pi^3$ such that either
$$\pi^1 r_{n-3}r_{n-1}=[2\,3\,n\,n-1\dots\,5\,4\,1]=\pi^3,$$
or
$$\pi^1 r_{n-1}r_{n-3}=[n-2\,n-3\dots\,3\,2\,n-1\,n\,1]=\pi^3,$$
and since $\pi^3=\gamma^4\,r_n$, then we have:

\begin{equation}\label{gamma1CD+}
\gamma^4 =
  \begin{cases}
     [1\,4\,5\,\dots\,n-1\,n\,3\,2]=\gamma^4(C)&or\\
     [1\,n\,n-1\,2\,3\,\dots\,n-3\,n-2]=\gamma^4(D).\\
  \end{cases}
\end{equation}

To get a sought $10$-cycle there should be a path of length three between $\gamma^1$ and $\gamma^4$. Let us check this. If $n=5$ then by~(\ref{gamma1AB+}) and~(\ref{gamma1CD+}) we have $\gamma^1(A)=[5\,1\,2\,4\,3]$ and $\gamma^4(D)=[1\,5\,4\,2\,3]$,  but there is no a path of length three between them, since we have either
$$\gamma^4(D)\,r_2\,r_4\,r_2=[4\,2\,1\,5\,3]\neq \gamma^1(A)$$
or
$$\gamma^4(D)\,r_4\,r_2\,r_4=[1\,5\,2\,4\,3]\neq \gamma^1(A).$$
If $n\geqslant 7$ then by~(\ref{gamma1AB+}) and~(\ref{gamma1CD+}) there is a contradiction with an assumption that both $\gamma^1$ and $\gamma^4$ belong to either $P_{n-1}(n-1)$ or $P_{n-1}(3)$. Hence, a sought cycle cannot occur in this case.

\vspace{5mm}
{\bf Case 3: 10-cycle within $P_n$ has vertices from four copies of $P_{n-1}$}
\vspace{5mm}

There are two possible situations in this case.

\vspace{3mm}

{\bf \underline{Case $(2+2+2+4).$}} Suppose that two vertices $\pi^1,\pi^2$ of a sought $10$--cycle belong to the first copy, two vertices $\tau^1,\tau^2$ belong to the second copy, two vertices $\gamma^1,\gamma^2$ belong to the third copy and remaining four vertices $\sigma^1,\sigma^2, \sigma^3, \sigma^4$ belong to the fourth copy (see Figure~4).

\begin{figure}[ht]\centering
\includegraphics[scale=0.4]{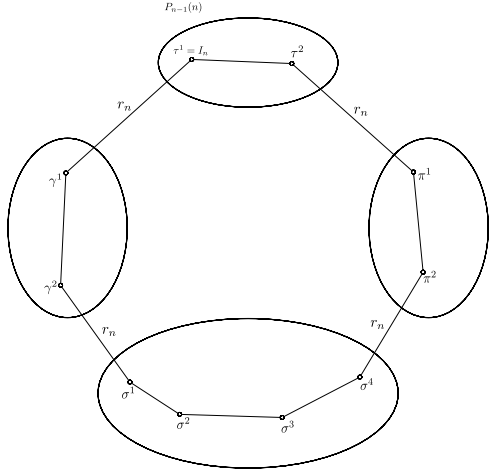}
\caption{Case 3: 10-cycle within $P_n$ has vertices from four copies of $P_{n-1}$}\label{4copy(2224)}
\end{figure}

Let $\pi^1=\tau^2 r_n$, $\pi^2=\sigma^4 r_n$, $\tau^1=\gamma^1 r_n$ and $\gamma^2=\sigma^1 r_n$, then both $\sigma^1$ and $\sigma^4$ should belong to either $P_{n-1}(4)$ or $P_{n-1}(2)$, since either $\sigma^1=\tau^1 r_n r_{n-1}r_n$ or $\sigma^1=\tau^1 r_n r_{n-3} r_n$. More precisely, we have:

\begin{equation}\label{sigma1AB}
\sigma^1 =
  \begin{cases}
     [1\,n\,n-1\,n-2\,\dots\,3\,2]=\sigma^1(A)&or\\
     [1\,2\,3\,n\,n-1\,n-2\,\dots\,6\,5\,4]=\sigma^1(B).\\
  \end{cases}
\end{equation}
On the other hand, similar to Case~$(2+2+6)$ there are four ways to reach $\sigma^4$ by a path of length four from $\tau^1$ such that we have:
\begin{equation}\label{sigma1CDEF}
\sigma^4 =
  \begin{cases}
     [n-3\,n-4\,n-5\,n\,n-1\,n-2\,1\,2\,\dots\,n-7\,n-6]=\sigma^4(C)&or\\
     [n-1\,n-2\,n-3\,n\,1\,2\,\ldots\,n-5\,n-4]=\sigma^4(D)&or\\
     [n-3\,n\,n-1\,n-2\,1\,2\,\dots\,n-5\,n-4]=\sigma^4(E)&or\\
     [n-1\,n\,1\,2\,\dots\,n-3\,n-2]=\sigma^4(F).\\
    \end{cases}
\end{equation}

To get a sought $10$-cycle there should be a path of length three between $\sigma^1$ and $\sigma^4$. Let us check this. If $n=5$ then by~(\ref{sigma1AB}) we have $\sigma^1(B)=[1\,2\,3\,5\,4]$ and $\sigma^4(C)=I_n\,(r_{n-3}\,r_n)^2=[2\,1\,3\,5\,4]$, but there is no a path of length three between them, since we have either

$$\sigma^4(C)\,r_2\,r_4\,r_2=[3\,5\,2\,1\,4]\neq \sigma^1(B)$$
or
$$\sigma^4(C)\,r_4\,r_2\,r_4=[2\,1\,5\,3\,4]\neq \sigma^1(B).$$

If $n\geqslant 7$ then by~(\ref{sigma1AB}) and~(\ref{sigma1CDEF}) there is a contradiction with an assumption that both $\sigma^1$ and $\sigma^4$ belong to either $P_{n-1}(4)$ or $P_{n-1}(2)$. Hence, a sought cycle cannot occur in this case.

\begin{figure}[ht]\centering
\includegraphics[scale=0.5]{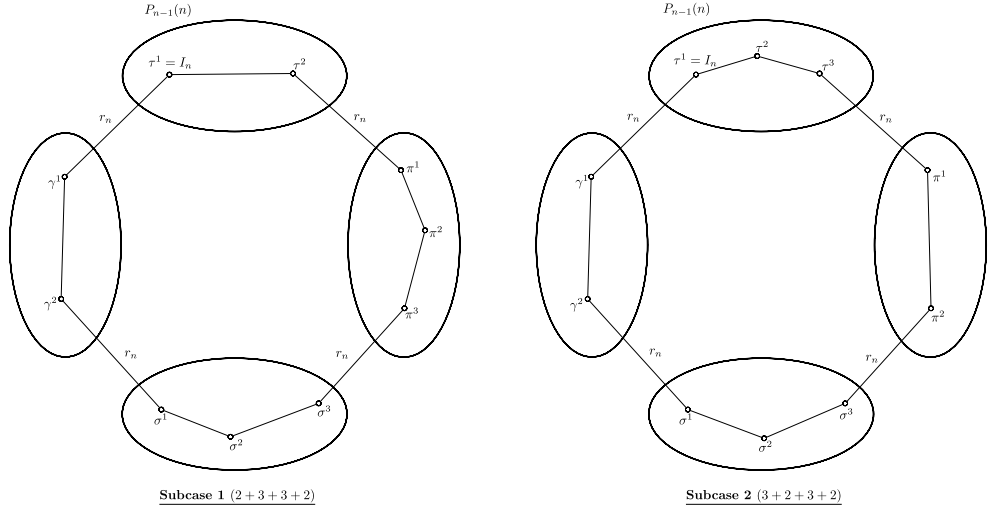}
\caption{Case 3: 10-cycle within $P_n$ has vertices from four copies of $P_{n-1}$}\label{4copy(2233)}
\end{figure}

{\bf \underline{Case $(2+2+3+3).$}}  There are two subcases due to a sequence of vertices from copies forming a cycle: 1) (2,3,3,2); 2) (3,2,3,2) (see Figure~5).

{\bf \underline{Subcase 1)}} Suppose that two vertices $\tau^1,\tau^2$ of a sought $10$--cycle belong to the first copy, three vertices $\pi^1,\pi^2,\pi^3$ belong to the second copy, three vertices $\sigma^1,\sigma^2, \sigma^3$ belong to the third copy and remaining two vertices $\gamma^1,\gamma^2$ belong to the fourth copy (see Figure~5, Subcase 1).

Let $\pi^1=\tau^2\,r_n$, $\sigma^3=\pi^3\,r_n$, and $\gamma^1=\tau^1\,r_n$, $\sigma^1=\gamma^2\,r_n$. Similar to Case~$(2+3+5)$ there are four ways to reach $\sigma^3$ by a path of length five from $\tau^1$ such that we have:
$$\sigma^3 =
  \begin{cases}
     [n-3\,n-6\,n-7\,\dots\,3\,2\,1\,n-2\,n-1\,n\,n-5\,n-4]=\sigma^3(A)&or\\
     [n-1\,n-4\,n-5\,\dots\,3\,2\,1\,n\,n-3\,n-2]=\sigma^3(B)&or\\
     [n-3\,n\,n-1\,n-4\,n-5\,\dots\,3\,2\,1\,n-2]=\sigma^3(C)&or\\
     [n-1\,n\,1\,n-2\,n-3\,\dots\,3\,2]=\sigma^3(D).\\
  \end{cases}$$
On the other hand, since $\gamma^1=\tau^1\,r_n$ and $\tau^1=I_n$ then $\gamma^1=[n\,n-1\,n-2\,\dots\,2\,1]$, and there are two ways to get a path of length one from $\gamma^1$ to $\gamma^2$ such that either
$$\gamma^1=[n\,n-1\,n-2\,\dots\,5\,4\,3\,2\,1]\xrightarrow{r_{n-3}}[4\,5\,\dots\,n-2\,n-1\,n\,3\,2\,1]=\gamma^2$$
or
$$\gamma^1=[n\,n-1\,n-2\,\dots\,3\,2\,1]\xrightarrow{r_{n-1}}[2\,3\,\dots\,n-2\,n-1\,n\,1]=\gamma^2,$$
and since $\sigma^1=\gamma^2\,r_n$ then we have either
$$\sigma^1=[1\,2\,3\,n\,n-1\,n-2\,\dots\,6\,5\,4]=\sigma^1(E)$$
or
$$\sigma^1=[1\,n\,n-1\,n-2\,\dots\,3\,2]=\sigma^1(F).$$

As one can see, vertices $\sigma^1(F)$ and $\sigma^3(D)$ belong to the same copy $P_{n-1}(2)$. Let us check whether there is a path of length two between these two vertices. Indeed, there are two ways to get a path of length two from $\sigma^1(F)$ to $\sigma^3(D)$. The first way is presented as follows:
\begin{equation}\label{sigma1_to_sigma3_1}
\sigma^1 r_{n-1}r_{n-3}=\sigma^3,
\end{equation}
where
$$\sigma^1(F)=[1\,n\,n-1\,\dots\,3\,2]\xrightarrow{r_{n-1}}[3\,4\,\dots\,n-1\,n\,1\,2]\xrightarrow{r_{n-3}}$$ $$[n-1\,n-2\,\dots\,4\,3\,n\,1\,2]\neq\sigma^3(D).$$
The second way is presented as follows:
\begin{equation}\label{sigma1_to_sigma3_2}
\sigma^1 r_{n-3}r_{n-1}=\sigma^3,
\end{equation}
where
$$\sigma^1(F)=[1\,n\,n-1\,\dots\,3\,2]\xrightarrow{r_{n-3}}[5\,6\,\dots\,n-1\,n\,1\,4\,3\,2]\xrightarrow{r_{n-1}}$$ $$[3\,4\,1\,n\,n-1\,\dots\,6\,5\,2]\neq\sigma^3(D).$$

Thus, a sought cycle cannot occur in this subcase.

{\bf \underline{Subcase 2.}} Suppose that three vertices $\tau^1,\tau^2,\tau^3$ of a sought $10$--cycle belong to the first copy, two vertices $\pi^1,\pi^2$ belong to the second copy, three vertices $\sigma^1,\sigma^2, \sigma^3$ belong to the third copy and remaining two vertices $\gamma^1,\gamma^2$  belong to the fourth copy. Let $\pi^1=\tau^3 r_n$, $\sigma^3=\pi^2 r_n$ and $\gamma^1=\tau^1 r_n=I_n$, $\sigma^1=\gamma^2 r_n$ (see Figure~5, Subcase 2). There are two ways to get a path of length two from $\tau^1$ to $\tau^3$. The first way is given as follows:
\begin{equation}\label{tau1_to_tau3_1}
\tau^1 r_{n-3}r_{n-1}=\tau^3,
\end{equation}
where
$$\tau^1=[1\,2\,3\,\dots\,n-1\,n]\xrightarrow{r_{n-3}}[n-3\,n-4\,\dots\,2\,1\,n-2\,n-1\,n]\xrightarrow{r_{n-1}}$$ $$[n-1\,n-2\,1\,2\,\dots\,n-4\,n-3\,n]=\tau^3.$$
The second way is given as follows:
\begin{equation}\label{tau1_to_tau3_2}
\tau^1 r_{n-1}r_{n-3}=\tau^3,
\end{equation}
where
$$\tau^1=[1\,2\,3\,\dots\,n-1\,n]\xrightarrow{r_{n-1}}[n-1\,n-2\,\dots\,2\,1\,n]\xrightarrow{r_{n-3}}$$ $$[3\,4\,\dots\,n-2\,n-1\,2\,1\,n]=\tau^3.$$

Since $\pi^1=\tau^3\,r_n$, then by~(\ref{tau1_to_tau3_1}) and~(\ref{tau1_to_tau3_2}) either $\pi^1=[n\,n-3\,n-4\,\dots\,2\,1\,n-2\,n-1]$ or $\pi^1=[n\,1\,2\,n-1\,n-2\,\dots\,4\,3]$, and since either $\pi^2=\pi^1\,r_{n-1}$ or $\pi^2=\pi^1\,r_{n-3}$ we have:
$$\pi^2 =
  \begin{cases}
     [2\,3\,\dots\,n-4\,n-3\,n\,1\,n-2\,n-1]=\pi^2(A)&or\\
     [6\,7\,\dots\,n-2\,n-1\,2\,1\,n\,5\,4\,3]=\pi^2(B)&or\\
     [n-2\,1\,2\,\dots\,n-4\,n-3\,n\,n-1]=\pi^2(C)&or\\
     [4\,5\,\dots\,n-2\,n-1\,2\,1\,n\,3]=\pi^2(D),\\
  \end{cases}$$
such that with $\sigma^3=\pi^2 r_n$ we obtain:
$$\sigma^3 =
  \begin{cases}
     [n-1\,n-2\,1\,n\,n-3\,n-4\,\dots\,3\,2]=\sigma^3(A)&or\\
     [3\,4\,5\,n\,1\,2\,n-1\,n-2\,\dots\,7\,6]=\sigma^3(B)&or\\
     [n-1\,n\,n-3\,n-4\,\dots\,2\,1\,n-2]=\sigma^3(C)&or\\
     [3\,n\,1\,2\,n-1\,n-2\,\dots\,5\,4]=\sigma^3(D).\\
  \end{cases}$$

On the other hand, there are two ways to get a path of length three from $\tau^1$ to $\sigma^1$ such that either
$$\sigma^1=\tau^1\,r_n\,r_{n-3}\,r_n=[1\,2\,3\,n\,\dots\,6\,5\,4]=\sigma^1(E)$$
or
$$\sigma^1=\tau^1\,r_n\,r_{n-1}\,r_n=[1\,n\,n-1\,\dots\,3\,2]=\sigma^1(F).$$
It is easy to see that vertices $\sigma^1(E)$ and $\sigma^3(D)$ belong to the copy $P_{n-1}(4)$. Let us check whether there is a path of length two between these two vertices. By~(\ref{sigma1_to_sigma3_1}) and~(\ref{sigma1_to_sigma3_2}), there are two ways to get a path of length two from $\sigma^1(E)$ to $\sigma^3(D)$ such that either
$$\sigma^1(E)=[1\,2\,3\,n\,\dots\,6\,5\,4]\xrightarrow{r_{n-1}}[5\,6\,\dots\,n-1\,n\,3\,2\,1\,4]\xrightarrow{r_{n-3}}$$ $$[3\,n\,n-1\,\dots\,6\,5\,2\,1\,4]\neq\sigma^3(D),$$
or
$$\sigma^1(E)=[1\,2\,3\,n\,\dots\,6\,5\,4]\xrightarrow{r_{n-3}}[7\,8\,\dots\,n-1\,n\,3\,2\,1\,6\,5\,4]\xrightarrow{r_{n-1}}$$ $$[5\,6\,1\,2\,3\,n\,n-1\,\dots\,8\,7\,4]\neq\sigma^3(D).$$

Hence, there is no a path of length two between these two vertices.

One can also see that vertices $\sigma^1(F)$ and $\sigma^3(A)$ belong to the copy $P_{n-1}(2)$. Let us check whether there is a path of length two between these two vertices. By~(\ref{sigma1_to_sigma3_1}) and~(\ref{sigma1_to_sigma3_2}), there are two ways to get a path of length two from $\sigma^1(F)$ to $\sigma^3(A)$ such as either
$$\sigma^1(F)=[1\,n\,n-1\,\dots\,3\,2]\xrightarrow{r_{n-1}}[3\,4\,\dots\,n-1\,n\,1\,2]\xrightarrow{r_{n-3}}$$ $$[n-1\,n-2\,\dots\,4\,3\,n\,1\,2]\neq\sigma^3(A),$$
or
$$\sigma^1(F)=[1\,n\,n-1\,\dots\,3\,2]\xrightarrow{r_{n-3}}[5\,6\,\dots\,n-1\,n\,1\,4\,3\,2]\xrightarrow{r_{n-1}}$$ $$[3\,4\,1\,n\,n-1\,\dots\,6\,5\,2]\neq\sigma^3(A).$$

Hence, there is no a path of length two between these two vertices, and a sought cycle cannot occur in this subcase.

\vspace{5mm}
{\bf Case 4: 10-cycle within $P_n$ has vertices from five copies of $P_{n-1}$}
\vspace{5mm}

There is the only possible situation if each of five copies has two vertices.

\begin{figure}[ht]\centering
\includegraphics[scale=0.7]{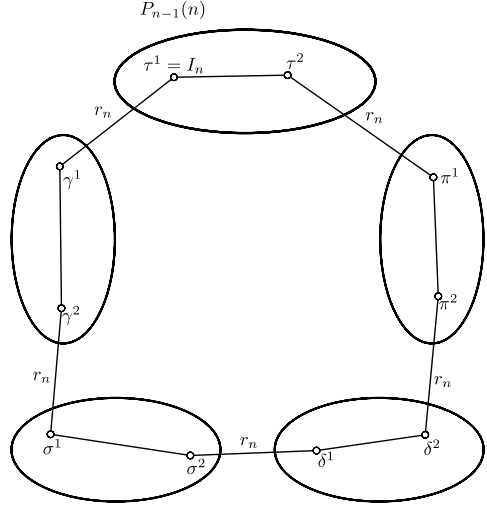}
\caption{Case 4: 10-cycle within $P_n$ has vertices from five copies of $P_{n-1}$}\label{5copy(22222)}
\end{figure}

{\bf \underline{Case $(2+2+2+2+2).$}} Suppose that vertices $\pi^1,\pi^2$ of a sought $10$--cycle belong to the first copy, vertices $\tau^1,\tau^2$ belong to the second copy, vertices $\gamma^1,\gamma^2$ belong to the third copy, vertices $\delta^1,\delta^2$ belong to the fourth copy and vertices $\sigma^1,\sigma^2$ belong to the fifth copy.

Let $\pi^1=\tau^2 r_n$, $\gamma^1=\tau^1\,r_n$, $\sigma^1=\gamma^2\,r_n$, $\delta^2=\pi^2\,r_n$ and $\delta^1=\sigma^2\,r_n$. Since $\tau^2$ can be reached from $\tau^1$  by either $r_{n-1}$ or $r_{n-3}$, and the same we have for $\pi^2$ and $\pi^1$, then there are four ways to get a path of length three from $\tau^1=I_n$ to $\pi^2$ (see Figure~6) such that:
$$\pi^2 =
  \begin{cases}
     [n-6\,n-7\,\dots\,2\,1\,n-2\,n-1\,n\,n-5\,n-4\,n-3]&or\\
     [n-4\,n-5\,\dots\,2\,1\,n\,n-3\,n-2\,n-1]&or\\
     [n-4\,n-5\,\dots\,2\,1\,n-2\,n-1\,n\,n-3]&or\\
     [n-2\,n-3\,\dots\,2\,1\,n\,n-1], \\
  \end{cases}$$
and since $\delta^2=\pi^2\,r_n$ we have:
\begin{equation}\label{delta2}
\delta^2 =
  \begin{cases}
     [n-3\,n-4\,n-5\,n\,n-1\,n-2\,1\,2\,\dots\,n-7\,n-6]&or\\
     [n-1\,n-2\,n-3\,n\,1\,2\,\dots\,n-5\,n-4]&or\\
     [n-3\,n\,n-1\,n-2\,1\,2\,\dots\,n-5\,n-4]&or\\
     [n-1\,n\,1\,2\,\dots\,n-3\,n-2].\\
  \end{cases}
\end{equation}
On the other hand, there are two ways to get a path of length two from $\tau^1$ to $\gamma^2$ such that either $$\tau^1\,r_n\,r_{n-3}=[4\,5\,\dots\,n-1\,n\,3\,2\,1]=\gamma^2$$
or
$$\tau^1\,r_n\,r_{n-1}=[2\,3\,\dots\,n-1\,n\,1]=\gamma^2,$$
and since $\sigma^1=\gamma^2 r_n$ we have:
$$\sigma^1 =
  \begin{cases}
     [1\,2\,3\,n\,\dots\,6\,5\,4]&or\\
     [1\,n\,n-1\,\dots\,3\,2],\\
  \end{cases}$$
and since either $\sigma^2=\sigma^1\,r_{n-1}$ or $\sigma^2=\sigma^1\,r_{n-3}$ we also have:
$$\sigma^2 =
  \begin{cases}
     [7\,8\,\dots\,n-1\,n\,3\,2\,1\,6\,5\,4]&or\\
     [5\,6\,\dots\,n-1\,n\,1\,4\,3\,2]&or\\
     [5\,6\,\dots\,n-1\,n\,3\,2\,1\,4]&or\\
     [3\,4\,\dots\,n-1\,n\,1\,2], \\
  \end{cases}$$
which gives us $\delta^1=\sigma^2\,r_n$ as follows:
\begin{equation}\label{delta1}
\delta^1 =
  \begin{cases}
     [4\,5\,6\,1\,2\,3\,n\,n-1\,\dots\,8\,7]&or\\
     [2\,3\,4\,1\,n\,n-1\,\dots\,6\,5]&or\\
     [4\,1\,2\,3\,n\,n-1\,\dots\,6\,5]&or\\
     [2\,1\,n\,n-1\,\dots\,4\,3].\\
  \end{cases}
\end{equation}

To get a sought $10$-cycle, either $\delta^1=\delta^2\,r_{n-3}$ or $\delta^1=\delta^2\,r_{n-1}$ should hold. Let us check this.

If $n=5$, then by~(\ref{delta2}) and~(\ref{delta1}) we have: \vspace{3mm}

\begin{minipage}{0.4\textwidth}
$$\delta^1 =
  \begin{cases}
     [4\,5\,3\,1\,2]=I_n(r_n\,r_{n-3})^2\,r_n&or\\
     [2\,3\,4\,1\,5]&or\\
     [4\,1\,2\,3\,5]&or\\
     [2\,1\,5\,4\,3],\\
  \end{cases}$$
\end{minipage}
\hspace{5mm}
\begin{minipage}{0.4\textwidth}
$$\delta^2 =
  \begin{cases}
     [4\,3\,2\,5\,1]&or\\
     [4\,5\,1\,2\,3]&or\\
     [2\,1\,3\,5\,4]&or\\
     [2\,5\,4\,3\,1].\\
  \end{cases}$$
\end{minipage}
\vspace{3mm}

As one can see, there are two vertices $[2\,1\,5\,4\,3]$ and $[4\,5\,1\,2\,3]$ belonging to the same copy, and there is the only way to get a sought $10$-cycle containing these vertices by the canonical form $C_{10}=(r_5\,r_4)^5$.

If $n=7$, then by~(\ref{delta2}) and~(\ref{delta1}) we have: \vspace{3mm}

\begin{minipage}{0.4\textwidth}
$$\delta^1 =
  \begin{cases}
     [4\,5\,6\,1\,2\,3\,7]&or\\
     [2\,3\,4\,1\,7\,6\,5]&or\\
     [4\,1\,2\,3\,7\,6\,5]&or\\
     [2\,1\,7\,6\,5\,4\,3],\\
  \end{cases}$$
\end{minipage}
\hspace{5mm}
\begin{minipage}{0.4\textwidth}
$$\delta^2 =
  \begin{cases}
     [4\,3\,2\,7\,6\,5\,1]&or\\
     [6\,5\,4\,7\,1\,2\,3]&or\\
     [4\,7\,6\,5\,1\,2\,3]&or\\
     [6\,7\,1\,2\,3\,4\,5].\\
  \end{cases}$$
\end{minipage}
\vspace{3mm}

\noindent It is evident that neither $[6\,5\,4\,7\,1\,2\,3]$ nor $[4\,7\,6\,5\,1\,2\,3]$ could not be reached from $[2\,1\,7\,6\,5\,4\,3]$ neither by $r_4$ nor by $r_6$. Thus, a sought $10$-cycle can not occur in this case. By using similar arguments, one can check that for $n=9,11,13$ there is no a $10$-cycle in the graph. For any odd $n\geqslant 15$, there is a contradiction with an assumption that both $\delta^1$ and $\delta^2$ belong to the same copy.

This complete the proof since all possible cases are considered. A $10$-cycle in the graph $P_n^5$ occurs only in the case when $n=5$ and with the canonical form $C_{10}=(r_5\,r_4)^5$.

\hfill $\square$

\end{proof}

\begin{theorem} \label{11-cycles} In the cubic Pancake graphs $P_n^{5}, n\geqslant 5,$ there are no cycles of length $11$.
\end{theorem}

\begin{proof} To prove this theorem, we use the same arguments as we used to prove Theorem~\ref{10-cycles}. Namely, we consider all possible cases for forming $11$--cycles in the Pancake graphs $P_n$ with taking into account that the generating set of $P_n^{5}$ contains only three elements $r_{n-3},r_{n-1},r_n$, where $n$ is odd. Due to the hierarchical structure of $P_n$, cycles of length $11$ could be formed from paths of length $l, \ 2\leqslant l\leqslant 9$, belonging to different $(n-1)-$copies of $P_n$. Further, we consider all possible options for the distribution of vertices by copies.

Within the proof without loss of generality we always put $\tau^1=I_n=[1\,2\,3\,\dots\,n-1\,n]$.

\vspace{5mm}
{\bf Case 1: $11$--cycle within $P_n$ has vertices from two copies of $P_{n-1}$}
\vspace{5mm}

Suppose that a sought $11$--cycle is formed on vertices from two different copies of $P_{n-1}$. By~\cite[Lemma 2]{KM10}, such a cycle cannot occur if its two (three) vertices belong to one copy and nine (eight) vertices belong to another one. Therefore, a sought cycle must have at least four vertices in each of the two copies. Hence, there are two following cases.

\vspace{3mm}

{\bf \underline{Case $(4+7).$}} Suppose that four vertices $\pi^1,\pi^2,\pi^3,\pi^4$ of a sought $11$--cycle belong to one copy, and other seven vertices $\tau^1,\tau^2,\tau^3,\tau^4,\tau^5,\tau^6,\tau^7$ belong to another copy. Let $\pi^1=\tau^1 r_n$, $\pi^4=\tau^7 r_n$, then $\pi^1$ and $\pi^4$ should belong to $P_{n-1}(1)$. Herewith, the four vertices of $P_{n-1}(1)$ should form a path of length three whose endpoints should be adjacent to vertices from $P_{n-1}(n)$. On the other hand, it is obvious that there are only two ways to get path of length six from $\tau^1$ to $\tau^7$, namely, either:
$$\tau^1\,(r_{n-3}\,r_{n-1})^3=[n-5\,n-6\,n-3\,n-4\,n-1\,n-2\,1\,2\,\dots\,n-8\,n-7\,n]=\tau^7$$
or
$$\tau^1\,(r_{n-1}\,r_{n-3})^3=[7\,8\,\dots\,n-2\,n-1\,2\,1\,4\,3\,6\,5\,n]=\tau^7.$$

Since $\pi^4=\tau^7\,r_n$ then we have either:
$$\pi^4=[n\,n-7\,n-8\dots\,2\,1\,n-2\,n-1\,n-4\,n-3\,n-6\,n-5]$$
or
$$\pi^4=[n\,5\,6\,3\,4\,1\,2\,n-1\,n-2\,\dots\,8\,7].$$

Note that $\pi^1=\tau^1\,r_n=[n\,n-1\,\dots\,2\,1]$ with $\pi_n^1=1$ for any $n\geqslant 5$. Hence, we immediately can conclude that $\pi^4$ and $\pi^1$ belong to different copies of $P_n$ since $\pi_n^4\neq 1$ for any odd $n\geqslant 5$. This gives a contradiction with an assumption that $\pi^1$ and $\pi^4$ belong to the same copy $P_{n-1}(1)$. Thus, a sought $11$--cycle cannot occur in this case.

\vspace{3mm}

{\bf \underline{Case $(5+6).$}} Suppose that five vertices $\pi^1,\pi^2,\pi^3,\pi^4,\pi^5$ of a sought $11$--cycle belong to copy $P_{n-1}(1)$, and other six vertices $\tau^1,\tau^2,\tau^3,\tau^4,\tau^5,\tau^6$ belong to copy $P_{n-1}(n)$, where $\pi^1=\tau^1 r_n$, and $\pi^5=\tau^6 r_n$. Herewith, the five vertices of $P_{n-1}(1)$ should form a path of length four whose endpoints should be adjacent to vertices from $P_{n-1}(n)$. By~(\ref{tau1_to_tau6_1}) and~(\ref{tau1_to_tau6_2}), there are two ways to get paths of length five from $\tau^1$ to $\tau^6$. Moreover, since $\pi^5=\tau^6\,r_n$ then we have either:
$$\pi^5=[n\,n-5\,n-6\,n-3\,n-4\,n-1\,n-2\,1\,2\,\dots\,n-8\,n-7]$$
or
$$\pi^5=[n\,5\,6\,\dots\,n-2\,n-1\,2\,1\,4\,3].$$

Since $\pi^1=\tau^1\,r_n=[n\,n-1\,\dots\,2\,1]$ with $\pi_n^1=1$ for any $n\geqslant 5$, then we immediately can conclude that $\pi^5$ and $\pi^1$ belong to different copies of $P_n$ since $\pi_n^5\neq 1$ for any odd $n\geqslant 5$. This gives a contradiction with an assumption that $\pi^1$ and $\pi^5$ belong to the same copy $P_{n-1}(1)$. Thus, a sought $11$--cycle cannot occur in this case.

\vspace{5mm}
{\bf Case 2: 11-cycle within $P_n$ has vertices from three copies of $P_{n-1}$}
\vspace{5mm}

There are five different situations in this case.

\vspace{3mm}

{\bf \underline{Case $(2+2+7).$}} Suppose that two vertices $\pi^1,\pi^2$ of a sought $11$--cycle belong to one copy, other two vertices $\tau^1,\tau^2$ belong to another copy, and remaining seven vertices $\gamma^1,\gamma^2, \gamma^3, \gamma^4, \gamma^5, \gamma^6,\gamma^7$ belong to the third copy. Let $\pi^1=\tau^2 r_n$, $\pi^2=\gamma^7 r_n$, $\tau^1=\gamma^1 r_n$, then $\gamma^1$ and $\gamma^7$ should belong to $P_{n-1}(1)$. Using similar reasoning shown in the proof of Theorem~\ref{10-cycles}, Case $(2+2+6)$, one can conclude that there are four ways to reach $\gamma^7$ by a path of length four from $\tau^1=I_n$ such that we have:
$$\gamma^7 =
  \begin{cases}
     [n-3\,n-4\,n-5\,n\,n-1\,n-2\,1\,2\,\dots\,n-7\,n-6]=\gamma^7(A)&or\\
     [n-1\,n-2\,n-3\,n\,1\,2\,\ldots\,n-5\,n-4]=\gamma^7(B)&or\\
     [n-3\,n\,n-1\,n-2\,1\,2\,\dots\,n-5\,n-4]=\gamma^7(C)&or\\
     [n-1\,n\,1\,2\,\dots\,n-3\,n-2]=\gamma^7(D).\\
  \end{cases}$$

To get a sought $11$-cycle there should be a path of length six between $\gamma^7$ and $\gamma^1$, where $\gamma^1=\tau^1\,r_n=[n\,n-1\,\dots\,2\,1]$. Let us check this. If $n=5$ then the vertices $\gamma^7(B)=[4\,3\,2\,5\,1]$, $\gamma^7(C)=[2\,5\,4\,3\,1]$ and $\gamma^1=[5\,4\,3\,2\,1]$ belong to the copy $P_{n-1}(1)$, and the following cases are possible:
$$\gamma^7(B)(r_2r_4)^3=[2\,5\,3\,4\,1]\neq \gamma^1$$
or
$$\gamma^7(C)(r_2r_4)^3=[4\,3\,5\,2\,1]\neq \gamma^1,$$
and
$$\gamma^7(B)(r_4r_2)^3=[5\,2\,4\,3\,1]\neq \gamma^1$$
or
$$\gamma^7(C)(r_4r_2)^3=[3\,4\,2\,5\,1]\neq \gamma^1.$$
Hence, a path of length five does not occur between $\gamma^7$ and $\gamma^1$.

If $n=7$ the vertices $\gamma^7(A)=[4\,3\,2\,7\,6\,5\,1]$ and $\gamma^1=[7\,6\,5\,4\,3\,2\,1]$ belong to the copy $P_{n-1}(1)$, and the following two cases are possible:
$$\gamma^7(A)(r_{4}\,r_{6})^3=[3\,4\,7\,2\,5\,6\,1]\neq \gamma^1$$
or
$$\gamma^7(A)(r_{6}\,r_{4})^3=[4\,3\,7\,2\,5\,6\,1]\neq \gamma^1.$$
Again, a path of length five does not occur between $\gamma^7$ and $\gamma^1$.

If $n\geqslant 9$, there is a contradiction with an assumption that $\gamma^1$ and $\gamma^7$ belong to the copy $P_{n-1}(1)$. Thus, a sought cycle cannot occur in this case.

{\bf \underline{Case $(2+3+6).$}} Suppose that two vertices $\tau^1,\tau^2$ of a sought $11$--cycle belong to the first copy, other three vertices $\pi^1,\pi^2,\pi^3$ belong to the second copy, and remaining six vertices $\gamma^1,\gamma^2, \gamma^3, \gamma^4, \gamma^5,\gamma^6$ belong to the third copy. Let $\pi^1=\tau^2 r_n$, $\pi^3=\gamma^6 r_n$, $\tau^1=\gamma^1 r_n$, then $\gamma^1$ and $\gamma^6$ should belong to $P_{n-1}(1)$. Taking into account similar reasoning used in the proof of Theorem~\ref{10-cycles}, {Case $(2+3+5)$}, one can conclude that there are four ways to reach $\gamma^6$ by a path of length five from $\tau^1=I_n$ such that we have:
$$\gamma^6 =
  \begin{cases}
     [n-3\,n-6\,n-7\,\dots2\,1\,n-2\,n-1\,n\,n-5\,n-4]=\gamma^6(A)&or\\
     [n-1\,n-4\,\dots\,3\,2\,1\,n\,n-3\,n-2]=\gamma^6(B)&or\\
     [n-3\,n\,n-1\,n-4\,\dots\,3\,2\,1\,n-2]=\gamma^6(C)&or\\
     [n-1\,n\,\,1\,n-2\,\dots\,3\,2]=\gamma^6(D).\\
  \end{cases}$$

To get a sought $11$-cycle there should be a path of length five between $\gamma^6$ and $\gamma^1$, where $\gamma^1=\tau^1\,r_n=[n\,n-1\,\dots\,2\,1]$. Let us check this. If $n=5$ the vertices $\gamma^6(A)=[2\,4\,5\,3\,1]$ and $\gamma^1=[5\,4\,3\,2\,1]$ belong to the copy $P_{n-1}(1)$, and the following cases are possible:
$$\gamma^6(A)(r_{n-3}\,r_{n-1})^2\,r_{n-3}=[2\,4\,5\,3\,1](r_2\,r_4)^2r_{2}=[2\,4\,3\,5\,1]\neq \gamma^1$$
or
$$\gamma^6(A)(r_{n-1}\,r_{n-3})^2\,r_{n-1}=[2\,4\,5\,3\,1](r_4\,r_2)^2r_{4}=[5\,3\,2\,4\,1]\neq \gamma^1.$$
Hence, a path of length five does not occur between $\gamma^6$ and $\gamma^1$.

If $n\geqslant 7$, there is a contradiction with an assumption that $\gamma^1$ and $\gamma^6$ belong to the copy $P_{n-1}(1)$. Thus, a sought cycle cannot occur in this case.

{\bf \underline{Case $(2+4+5).$}} Suppose that two vertices $\tau^1,\tau^2$ of a sought $11$--cycle belong to the first copy, other four vertices $\pi^1,\pi^2,\pi^3,\pi^4$ belong to the second copy, and remaining five vertices $\gamma^1,\gamma^2, \gamma^3, \gamma^4, \gamma^5$ belong to the third copy. Let $\pi^1=\tau^1 r_n$, $\pi^4=\gamma^5 r_n$, $\tau^2=\gamma^1 r_n$, then $\gamma^1$ and $\gamma^5$ should belong to either $P_{n-1}(n-1)$ or $P_{n-1}(n-3)$, since either $\gamma^1=\tau^1 r_{n-1}r_n$ or $\gamma^1=\tau^1 r_{n-3}r_n$, where:
\begin{equation}\label{gamma1AB_new}
\gamma^1 =
  \begin{cases}
     [n\,1\,2\,\dots\,n-2\,n-1]=\gamma^1(A)&or\\
     [n\,n-1\,n-2\,1\,2\,\dots\,n-4\,n-3]=\gamma^1(B).\\
  \end{cases}
\end{equation}

By the same reasoning used in the proof of Theorem~\ref{10-cycles}, {Case $(2+4+4)$}, one can see that there are two ways to reach $\gamma^5$ by a path of length five from $\tau^1=I_n$ such that we have:
\begin{equation}\label{gamma5CD_new}
\gamma^5 =
  \begin{cases}
     [1\,4\,5\,2\,3\,n\,n-1\dots\,7\,6]=\gamma^5(C)& or\\
     [1\,n-2\,\dots\,3\,2\,n-1\,n]=\gamma^5(D).\\
  \end{cases}
\end{equation}

To get a sought $11$-cycle there should be a path of length four between $\gamma^1$ and $\gamma^5$. Let us check this. If $n=5$ or $n\geqslant 11$, there is a contradiction with an assumption that both $\gamma^1$ and $\gamma^5$ belong to either $P_{n-1}(n-1)$ or $P_{n-1}(n-3)$. If $n=7$ then by~(\ref{gamma1AB_new}) and~(\ref{gamma5CD_new}) we have $\gamma^1(A)=[7\,1\,2\,3\,4\,5\,6]$ and $\gamma^5(C)=[1\,4\,5\,2\,3\,7\,6]$, but there is no a path of length four between them, since we have either
$$\gamma^5(C)(r_4\,r_6)^2=[2\,5\,7\,3\,1\,4\,6]\neq \gamma^1(A)$$
or
$$\gamma^5(C)(r_6\,r_4)^2=[3\,7\,4\,1\,2\,5\,6]\neq \gamma^1(A).$$

If $n=9$ then by~(\ref{gamma1AB_new}) and~(\ref{gamma5CD_new}) we have $\gamma^1(B)=[9\,8\,7\,1\,2\,3\,4\,5\,6]$ and $\gamma^5(C)=[1\,4\,5\,2\,3\,9\,8\,7\,6]$, but there is no a path of length four between them, since we have either
$$\gamma^5(C)(r_6\,r_8)^2=[9\,3\,7\,8\,1\,4\,5\,2\,6]\neq \gamma^1(A)$$
or
$$\gamma^5(C)(r_8\,r_6)^2=[3\,9\,8\,7\,1\,4\,2\,5\,6]\neq \gamma^1(A).$$
Thus, a sought cycle cannot occur in this case.

{\bf \underline{Case $(3+3+5).$}} Suppose that three vertices $\tau^1,\tau^2,\tau^3$ of a sought $11$--cycle belong to the first copy, other three vertices $\pi^1,\pi^2,\pi^3$ belong to the second copy, and remaining five vertices $\gamma^1,\gamma^2, \gamma^3, \gamma^4,\gamma^5$ belong to the third copy. Let $\pi^1=\tau^1 r_n=[n\,n-1\,\dots\,2\,1]$, $\pi^3=\gamma^5 r_n$, $\tau^3=\gamma^1 r_n$, then $\gamma^1$ and $\gamma^5$ should belong to either $P_{n-1}(n-1)$ or $P_{n-1}(3)$, since either $\gamma^1=\tau^1\,r_{n-3}\,r_{n-1}\,r_n$ or $\gamma^1=\tau^1\,r_{n-1}\,r_{n-3}\,r_n$, where
\begin{equation}\label{gamma1AB_335}
\gamma^1 =
  \begin{cases}
     [n\,n-3\,n-4\,\dots\,2\,1\,n-2\,n-1]=\gamma^1(A)&or\\
     [n\,1\,2\,n-1\,n-2\,\dots\,4\,3]=\gamma^1(B).\\
  \end{cases}
\end{equation}
Taking into account the same arguments as we used in the proof of Theorem~\ref{10-cycles}, {Case $(3+3+4)$}, one can conclude that there are two ways to reach $\gamma^5$ by a path of length five from $\tau^1=I_n$ such that we have:
\begin{equation}\label{gamma5CD_335}
\gamma^5 =
  \begin{cases}
     [1\,4\,5\,\dots\,n-1\,n\,3\,2]=\gamma^5(C)&or\\
     [1\,n\,n-1\,2\,3\,\dots\,n-3\,n-2]=\gamma^5(D).\\
  \end{cases}
\end{equation}

To get a sought $11$-cycle there should be a path of length four between $\gamma^1$ and $\gamma^5$. Let us check this. If $n=5$ then by~(\ref{gamma1AB_335}) and~(\ref{gamma5CD_335}) we have $\gamma^1(B)=[5\,1\,2\,4\,3]$ and $\gamma^5(D)=[1\,5\,4\,2\,3]$,  but there is no a path of length four between them, since we have either
$$\gamma^5(D)(r_2\,r_4)^2=[5\,1\,2\,4\,3]\neq \gamma^1(B)$$
or
$$\gamma^5(D)(r_4\,r_2)^2=[5\,1\,2\,4\,3]\neq \gamma^1(B).$$
Hence, a path of length four does not occur between $\gamma^5$ and $\gamma^1$. However, let us note that in this case we have a cycle of length eight given by the canonical form $C_8=(r_4\,r_2)^4$ obtained from~(\ref{C87}) by putting $k=4,\,j=3,\,i=2$.

If $n\geqslant 7$ then by~(\ref{gamma1AB_335}) and~(\ref{gamma5CD_335}) there is a contradiction with an assumption that both $\gamma^1$ and $\gamma^5$ belong to either $P_{n-1}(n-1)$ or $P_{n-1}(3)$. Hence, a sought cycle cannot occur in this case.

{\bf \underline{Case $(3+4+4).$}} Suppose that three vertices $\tau^1,\tau^2,\tau^3$ of a sought $11$--cycle belong to the first copy, other four vertices $\pi^1,\pi^2,\pi^3,\pi^4$ belong to the second copy, and remaining four vertices $\gamma^1,\gamma^2, \gamma^3, \gamma^4$ belong to the third copy. Let $\pi^1=\tau^1 r_n=[n\,n-1\,\dots\,2\,1]$, $\pi^4=\gamma^4 r_n$, $\tau^3=\gamma^4 r_n$, then $\gamma^1$ and $\gamma^4$ should belong to either $P_{n-1}(n-1)$ or $P_{n-1}(3)$, since either $\gamma^1=\tau^1 r_{n-1}r_{n-3}r_n$ or $\gamma^1=\tau^1 r_{n-3}r_{n-1}r_n$, where
\begin{equation}\label{gamma1AB_344}
\gamma^1 =
  \begin{cases}
     [n\,n-3\,n-4\,\dots\,2\,1\,n-2\,n-1]=\gamma^1(A) &or\\
     [n\,1\,2\,n-1\,n-2\,\dots\,4\,3]=\gamma^1(B) .\\
  \end{cases}
\end{equation}
On the other hand, it is obvious that there are two ways to get a path of length five from $\tau^1$ to $\gamma^5$:
\begin{equation}\label{gamma4CD_344}
\gamma^4 =
  \begin{cases}
     \tau^1\,r_n\,r_{n-3}\,r_{n-1}\,r_{n-3}\,r_n=[1\,4\,5\,2\,3\,n\,n-1\dots\,7\,6]=\gamma^4(C) & or\\
     \tau^1\,r_n\,r_{n-1}\,r_{n-3}\,r_{n-1}\,r_n=[1\,n-2\,\dots\,3\,2\,n-1\,n]=\gamma^4(D) .\\
  \end{cases}
\end{equation}

To get a sought $11$-cycle there should be a path of length four between $\gamma^1$ and $\gamma^4$. Let us check this. If $n=5$ then by~(\ref{gamma1AB_344}) and~(\ref{gamma4CD_344}) we have $\gamma^1(B)=[5\,1\,2\,4\,3]$ and $\gamma^4(C)=[1\,4\,5\,2\,3]$,  but there is no a path of length three between them, since we have either
$$\gamma^4(C)\,r_2\,r_4\,r_2=[5\,2\,1\,4\,3]\neq \gamma^1(B)$$
or
$$\gamma^4(C)\,r_4\,r_2\,r_4=[1\,4\,2\,5\,3]\neq \gamma^1(B).$$
Hence, a path of length three does not occur between $\gamma^4$ and $\gamma^1$.

If $n=7$ then by~(\ref{gamma1AB_344}) and~(\ref{gamma4CD_344}) we have $\gamma^1(A)=[7\,4\,3\,2\,1\,5\,6]$ and $\gamma^4(C)=[1\,4\,5\,2\,3\,7\,6]$,  but there is no a path of length three between them, since we have either
$$\gamma^4(C)\,r_4\,r_6\,r_4=[4\,1\,3\,7\,5\,2\,6]\neq \gamma^1(A)$$
or
$$\gamma^4(C)\,r_6\,r_4\,r_6=[1\,4\,7\,3\,2\,5\,6]\neq \gamma^1(A).$$
Again, a path of length three does not occur between $\gamma^4$ and $\gamma^1$.

If $n\geqslant 9$ then there is a contradiction with an assumption that both $\gamma^1$ and $\gamma^4$ belong to either $P_{n-1}(n-1)$ or $P_{n-1}(3)$. Hence, a sought cycle cannot occur in this case.

\vspace{5mm}
{\bf Case 3: 11-cycle within $P_n$ has vertices from four copies of $P_{n-1}$}
\vspace{5mm}

There are three possible situations in this case.

\vspace{3mm}

{\bf \underline{Case $(2+2+2+5).$}} Suppose that two vertices $\pi^1,\pi^2$ of a sought $11$--cycle belong to the first copy, two vertices $\tau^1,\tau^2$ belong to the second copy, two vertices $\gamma^1,\gamma^2$ belong to the third copy and remaining five vertices $\sigma^1,\sigma^2, \sigma^3, \sigma^4,\sigma^5$ belong to the fourth copy. Let $\tau^1=\gamma^1 r_n$, $\gamma^2=\sigma^1 r_n$, $\sigma^5=\pi^2 r_n$ and $\pi^1=\tau^2 r_n$, then $\sigma^1$ and $\sigma^5$ should belong to either $P_{n-1}(2)$ or $P_{n-1}(4)$, since either $\sigma^1=\tau^1 r_n r_{n-1}r_n$ or $\sigma^1=\tau^1 r_n r_{n-3} r_n$ where
\begin{equation}\label{sigma1AB_2225}
\sigma^1 =
  \begin{cases}
     [1\,n\,n-1\,n-2\,\dots\,3\,2]=\sigma^1(A)&or\\
     [1\,2\,3\,n\,n-1\,\dots\,5\,4]=\sigma^1(B).\\
  \end{cases}
\end{equation}
Taking into account the same arguments as we used in the proof of Theorem~\ref{10-cycles}, {Case $(2+2+2+4)$}, one can conclude that there are four ways to reach $\sigma^5$ by a path of length four from $\tau^1=I_n$ such that we have:

\begin{equation}\label{sigma1CDEF_2225}
\sigma^5 =
  \begin{cases}
     [n-3\,n-4\,n-5\,n\,n-1\,n-2\,1\,2\,\dots\,n-7\,n-6]=\sigma^5(C)&or\\
     [n-1\,n-2\,n-3\,n\,1\,2\,\ldots\,n-5\,n-4]=\sigma^5(D)&or\\
     [n-3\,n\,n-1\,n-2\,1\,2\,\dots\,n-5\,n-4]=\sigma^5(E)&or\\
     [n-1\,n\,1\,2\,\dots\,n-3\,n-2]=\sigma^5(F).\\
    \end{cases}
\end{equation}

To get a sought $11$-cycle there should be a path of length four between $\sigma^1$ and $\sigma^5$. Let us check this. If $n=5$ then by~(\ref{sigma1AB_2225}) we have $\sigma^1(B)=[1\,2\,3\,5\,4]$, and $\sigma^5(C)=I_n(r_{n-3}\,r_{n})^2=[2\,1\,3\,5\,4]$, but there is no a path of length four between them, since we have either

$$\sigma^5(C)(r_2\,r_4)^2=[1\,2\,5\,3\,4]\neq \sigma^1(B)$$
or
$$\sigma^5(C)(r_4\,r_2)^2=[2\,1\,5\,3\,4]\neq \sigma^1(B).$$

If $n\geqslant 7$ then by~(\ref{sigma1AB_2225}) and~(\ref{sigma1CDEF_2225}) there is a contradiction with an assumption that both $\sigma^1$ and $\sigma^5$ belong to either $P_{n-1}(2)$ or $P_{n-1}(4)$. Hence, a sought cycle cannot occur in this case.

{\bf \underline{Case $(2+3+3+3).$}} Suppose that three vertices $\tau^1,\tau^2,\tau^3$ of a sought $11$--cycle belong to the first copy, three vertices $\pi^1,\pi^2,\pi^3$ belong to the second copy, three vertices $\gamma^1,\gamma^2,\gamma^3$ belong to the third copy and remaining two vertices $\sigma^1,\sigma^2$ belong to the fourth copy. Let $\tau^1=\pi^1 r_n$, $\pi^3=\sigma^2 r_n$, $\tau^3=\gamma^1 r_n$ and $\gamma^3=\sigma^1 r_n$. Taking into account the same arguments as we used in the proof of Theorem~\ref{10-cycles}, {Case $(3+3+4)$}, one can conclude that there are two ways to reach $\sigma^2$ by a path of length three from $\tau^1=I_n$ such that we have:
\begin{equation}\label{gamma2AB_2333}
\sigma^2 =
  \begin{cases}
     \tau^1\,r_n\,r_{n-3}\,r_{n-1}\,r_{n-3}\,r_n=[1\,4\,5\,2\,3\,n\,n-1\dots\,7\,6]=\sigma^2(A) & or\\
     \tau^1\,r_n\,r_{n-1}\,r_{n-3}\,r_{n-1}\,r_n=[1\,n-2\,\dots\,3\,2\,n-1\,n]=\sigma^2(B) .\\
  \end{cases}
\end{equation}

On the other hand, by~(\ref{tau1_to_tau3_1}) and~(\ref{tau1_to_tau3_2}), and since $\gamma_1=\tau^3 r_n$, there are two ways to get paths of length three from $\tau^1$ to $\gamma^1$ such that either
$$\gamma^1=[n\,n-3\,n-4\,\dots\,2\,1\,n-2\,n-1]$$
or
$$\gamma^1=[n\,1\,2\,n-1\,n-2\,\dots\,4\,3].$$

Then there are two ways to get paths of length two from $\gamma^1$ to $\gamma^3$ such that the first way is presented as follows:
$$\gamma^1 r_{n-3}r_{n-1}=\gamma^3.$$
Namely, we get either
$$\gamma^1=[n\,n-3\,n-4\,\dots\,2\,1\,n-2\,n-1]\xrightarrow{r_{n-3}}[2\,3\,\dots\,n-4\,n-3\,n\,1\,n-2\,n-1]\xrightarrow{r_{n-1}}$$
$$[n-2\,1\,n\,n-3\,n-4\,\dots\,3\,2\,n-1]=\gamma^3$$
or
$$\gamma^1=[n\,1\,2\,n-1\,n-2\,\dots\,4\,3]\xrightarrow{r_{n-3}}[6\,7\,\dots\,n-2\,n-1\,2\,1\,n\,5\,4\,3]\xrightarrow{r_{n-1}}$$
$$[4\,5\,n\,1\,2\,n-1\,n-2\,\dots\,6\,5\,3]=\gamma^3.$$
The second way is presented as follows:
$$\gamma^1 r_{n-1}r_{n-3}=\gamma^3.$$
Namely, we get
$$\gamma^1=[n\,n-3\,n-4\,\dots\,2\,1\,n-2\,n-1]\xrightarrow{r_{n-1}}[n-2\,1\,2\,\dots\,n-4\,n-3\,n\,n-1]\xrightarrow{r_{n-3}}$$
$$[n-4\,n-5\,\dots\,2\,1\,n-2\,n-3\,n\,n-1]=\gamma^3$$
or
$$\gamma^1=[n\,1\,2\,n-1\,n-2\,\dots\,4\,3]\xrightarrow{r_{n-1}}[4\,5\,\dots\,n-2\,n-1\,2\,1\,n\,3]\xrightarrow{r_{n-3}}$$
$$[2\,n-1\,n-2\,\dots\,5\,4\,1\,n\,3]=\gamma^3.$$

Since $\sigma^1=\gamma^3r_n$, we get
\begin{equation}\label{sigma1CDEF_2333}
\sigma^1 =
  \begin{cases}
     [n-1\,2\,3\,\dots\,n-4\,n-3\,n\,1\,n-2]=\sigma^1(C)&or\\
     [3\,5\,6\,\dots\,n-2\,n-1\,2\,1\,n\,5\,4]=\sigma^1(D)&or\\
     [n-1\,n\,n-3\,n-2\,1\,2\,\dots\,n-5\,n-4]=\sigma^1(E)&or\\
     [3\,n\,1\,4\,5\,\dots\,n-1\,2]=\sigma^1(F).\\
  \end{cases}
\end{equation}

To get a sought $11$-cycle vertices $\sigma^1$ and $\sigma^2$ should be adjacent by an internal edge. However, if $n=5$ then by~(\ref{sigma1CDEF_2333}) and~(\ref{gamma2AB_2333}),  we have two non-adjacent vertices $\sigma^1(C)=[4\,2\,5\,1\,3]$ and $\sigma^2(A)=[1\,4\,5\,2\,3]$. If $n\geqslant 7$ then by~(\ref{gamma2AB_2333}) and~(\ref{sigma1CDEF_2333}) there is a contradiction with an assumption that $\sigma^1$ and $\sigma^2$ belong to the same copy. Hence, a sought cycle cannot occur in this case.

{\bf \underline{Case $(2+2+3+4).$}} There are two subcases due to a sequence of vertex from the copies forming a cycle: 1) (2,3,4,2); 2) (3,2,4,2).

{\bf \underline{Subcase 1.}} Suppose that four vertices $\pi^1,\pi^2,\pi^3,\pi^4$ of a sought $11$--cycle belong to the first copy, two vertices $\tau^1,\tau^2$ belong to the second copy, three vertices $\gamma^1,\gamma^2,\gamma^3$ belong to the third copy and remaining two vertices $\sigma^1,\sigma^2$ belong to the fourth copy. Taking into account the same arguments as we used in the proof of Theorem~\ref{10-cycles}, {Case $(2+3+2+3)$}, one can conclude that there are four ways to reach $\sigma^4$ by a path of length five from $\tau^1=I_n$:
$$\sigma^4 =
  \begin{cases}
     [n-3\,n-6\,n-7\,\dots\,3\,2\,1\,n-2\,n-1\,n\,n-5\,n-4]=\sigma^4(A)&or\\
     [n-1\,n-4\,n-5\,\dots\,3\,2\,1\,n\,n-3\,n-2]=\sigma^4(B)&or\\
     [n-3\,n\,n-1\,n-4\,n-5\,\dots\,3\,2\,1\,n-2]=\sigma^4(C)&or\\
     [n-1\,n\,1\,n-2\,n-3\,\dots\,3\,2]=\sigma^4(D).\\
  \end{cases}$$
On the other hand, there are two ways to reach $\sigma^1$ by a path of length three from $\tau^1$ such that we have:
$$\sigma^1 =
  \begin{cases}
     [1\,2\,3\,n\,n-1\,n-2\,\dots\,6\,5\,4]=\sigma^1(E)&or\\
     [1\,n\,n-1\,n-2\,\dots\,3\,2]=\sigma^1(F).\\
  \end{cases}$$

As one can see, vertices $\sigma^1(F)$ and $\sigma^4(D)$ belong to the same copy $P_{n-1}(2)$. Let us check whether there is a path of length three between these two vertices. Indeed, there are two ways to get a path of length three from $\sigma^1(F)$ to $\sigma^4(D)$. The first way is presented as follows:
\begin{equation}\label{sigma1_to_sigma4_1}
\sigma^1 r_{n-1}r_{n-3}r_{n-1}=\sigma^4,
\end{equation}
where
$$\sigma^1(F)=[1\,n\,n-1\,\dots\,3\,2]\xrightarrow{r_{n-1}}[3\,4\,\dots\,n-1\,n\,1\,2]\xrightarrow{r_{n-3}}$$ $$[n-1\,n-2\,\dots\,4\,3\,n\,1\,2]\xrightarrow{r_{n-1}}[1\,n\,3\,4\,\dots\,n-2\,n-1\,2]\neq\sigma^4(D).$$
The second way is presented as follows:
\begin{equation}\label{sigma1_to_sigma4_2}
\sigma^1 r_{n-3}r_{n-1}r_{n-3}=\sigma^4,
\end{equation}
where
$$\sigma^1(F)=[1\,n\,n-1\,\dots\,3\,2]\xrightarrow{r_{n-3}}[5\,6\,\dots\,n-1\,n\,1\,4\,3\,2]\xrightarrow{r_{n-1}}$$ $$[3\,4\,1\,n\,n-1\,\dots\,6\,5\,2]\xrightarrow{r_{n-3}}[7\,8\,\dots\,n-1\,n\,1\,4\,3\,6\,5\,2]\neq\sigma^4(D).$$

Thus, a sought cycle cannot occur in this subcase.

{\bf \underline{Subcase 2.}} Suppose that three vertices $\tau^1,\tau^2,\tau^3$ of a sought $11$--cycle belong to the first copy, two vertices $\pi^1,\pi^2$ belong to the second copy, four vertices $\sigma^1,\sigma^2, \sigma^3, \sigma^4$ belong to the third copy and remaining two vertices $\gamma^1,\gamma^2$  belong to the fourth copy. Let $\pi^1=\tau^3 r_n$, $\pi^2=\sigma^4 r_n$ and $\tau^1=\gamma^1 r_n$, $\gamma^2=\sigma^1 r_n$. Taking into account the same arguments as we used in the proof of Theorem~\ref{10-cycles}, {Case $(3+2+3+2)$}, one can conclude that there are four ways to reach $\sigma^4$ by a path of length five from $\tau^1=I_n$:
$$\sigma^4 =
  \begin{cases}
     [n-1\,n-2\,1\,n\,n-3\,n-4\,\dots\,3\,2]=\sigma^4(A)&or\\
     [3\,4\,5\,n\,1\,2\,n-1\,n-2\,\dots\,7\,6]=\sigma^4(B)&or\\
     [n-1\,n\,n-3\,n-4\,\dots\,2\,1\,n-2]=\sigma^4(C)&or\\
     [3\,n\,1\,2\,n-1\,n-2\,\dots\,5\,4]=\sigma^4(D).\\
  \end{cases}$$
On the other hand, there are two ways to reach $\sigma^1$ by a path of length three from $\tau^1$ such that we have:
$$\sigma^1 =
  \begin{cases}
     [1\,2\,3\,n\,\dots\,6\,5\,4]=\sigma^1(E)&or\\
     [1\,n\,n-1\,n-2\,\dots\,3\,2]=\sigma^1(F).\\
  \end{cases}$$

It is easy to see that vertices $\sigma^1(E)$ and $\sigma^4(D)$ belong to the copy $P_{n-1}(4)$. However, there is no a path of length three between these two vertices. Indeed, by~(\ref{sigma1_to_sigma4_1}) and~(\ref{sigma1_to_sigma4_2}), there are two ways to get a path of length three from $\sigma^1(E)$ to $\sigma^4(D)$ such that either
$$\sigma^1(E)=[1\,2\,3\,n\,\dots\,6\,5\,4]\xrightarrow{r_{n-1}}[5\,6\,\dots\,n-1\,n\,3\,2\,1\,4]\xrightarrow{r_{n-3}}$$ $$[3\,n\,n-1\,\dots\,6\,5\,2\,1\,4]\xrightarrow{r_{n-1}}[1\,2\,5\,6\,\dots\,n-1\,n\,3\,4]\neq\sigma^4(D),$$
or
$$\sigma^1(E)=[1\,2\,3\,n\,\dots\,6\,5\,4]\xrightarrow{r_{n-3}}[7\,8\,\dots\,n-1\,n\,3\,2\,1\,6\,5\,4]\xrightarrow{r_{n-1}}$$ $$[5\,6\,1\,2\,3\,n\,n-1\,\dots\,8\,7\,4]\xrightarrow{r_{n-3}}[9\,10\,\dots\,n-1\,n\,3\,2\,1\,6\,5\,8\,7\,4]\neq\sigma^4(D).$$

Hence, there is no a path of length three between these two vertices.

One can also see that vertices $\sigma^1(F)$ and $\sigma^4(A)$ belong to the copy $P_{n-1}(2)$. Let us check whether there is a path of length three between these two vertices. By~(\ref{sigma1_to_sigma4_1}) and~(\ref{sigma1_to_sigma4_2}), there are two ways to get a path of length three from $\sigma^1(F)$ to $\sigma^4(A)$ such as either
$$\sigma^1(F)=[1\,n\,n-1\,\dots\,3\,2]\xrightarrow{r_{n-1}}[3\,4\,\dots\,n-1\,n\,1\,2]\xrightarrow{r_{n-3}}$$ $$[n-1\,n-2\,\dots\,4\,3\,n\,1\,2]\xrightarrow{r_{n-1}}[1\,n\,3\,4\,\dots\,n-2\,n-1\,2]\neq\sigma^4(A),$$
or
$$\sigma^1(F)=[1\,n\,n-1\,\dots\,3\,2]\xrightarrow{r_{n-3}}[5\,6\,\dots\,n-1\,n\,1\,4\,3\,2]\xrightarrow{r_{n-1}}$$ $$[3\,4\,1\,n\,n-1\,\dots\,6\,5\,2]\xrightarrow{r_{n-3}}[7\,8\,\dots\,n-1\,n\,1\,4\,3\,6\,5\,2]\neq\sigma^4(A).$$

Thus, a sought cycle cannot occur in this subcase.

\vspace{5mm}
{\bf Case 4: 11-cycle within $P_n$ has vertices from five copies of $P_{n-1}$}
\vspace{5mm}

There is the only possible situation in this case.

{\bf \underline{Case $(2+2+2+2+3).$}} Suppose that two vertices $\pi^1,\pi^2$ of a sought $11$--cycle belong to the first copy, two vertices $\tau^1,\tau^2$ belong to the second copy, two vertices $\gamma^1,\gamma^2$ belong to the third copy, three vertices $\delta^1,\delta^2,\delta^3$ belong to the fourth copy and remaining two vertices $\sigma^1,\sigma^2$ belong to the fifth copy. Let $\pi^1=\tau^2 r_n$, $\pi^2=\delta^3 r_n$, $\delta^1=\sigma^2 r_n$, $\sigma^1=\gamma^2 r_n$ and $\gamma^1=\tau^1 r_n$. Taking into account the same arguments as we used in the proof of Theorem~\ref{10-cycles}, {Case $(2+2+2+2+2)$}, one can conclude that there are four ways to reach $\delta^3$ by a path of length four from $\tau^1=I_n$:
\begin{equation}\label{delta3_22223}
\delta^3 =
  \begin{cases}
     [n-3\,n-4\,n-5\,n\,n-1\,n-2\,1\,2\,\dots\,n-7\,n-6]&or\\
     [n-1\,n-2\,n-3\,n\,1\,2\,\dots\,n-5\,n-4]&or\\
     [n-3\,n\,n-1\,n-2\,1\,2\,\dots\,n-5\,n-4]&or\\
     [n-1\,n\,1\,2\,\dots\,n-3\,n-2].\\
  \end{cases}
\end{equation}

On the other hand, there are four ways to reach $\delta^1$ by a path of length five from $\tau^1$ such that we have:
\begin{equation}\label{delta1_22223}
\delta^1 =
  \begin{cases}
     [4\,5\,6\,1\,2\,3\,n\,n-1\,\dots\,8\,7]&or\\
     [2\,3\,4\,1\,n\,n-1\,\dots\,6\,5]&or\\
     [4\,1\,2\,3\,n\,n-1\,\dots\,6\,5]&or\\
     [2\,1\,n\,n-1\,\dots\,4\,3].\\
  \end{cases}
\end{equation}

To get a sought $11$-cycle, either $\delta^1=\delta^3\,r_{n-3}\,r_{n-1}$ or $\delta^1=\delta^3\,r_{n-1}\,r_{n-3}$ should hold. Let us check this. If $n=5$, then by~(\ref{delta1_22223}) and~(\ref{delta3_22223}) we have: \vspace{3mm}

\begin{minipage}{0.4\textwidth}
$$\delta^1 =
  \begin{cases}
     [4\,5\,3\,1\,2]=I_n(r_n\,r_{n-3})^2\,r_n&or\\
     [2\,3\,4\,1\,5]&or\\
     [4\,1\,2\,3\,5]&or\\
     [2\,1\,5\,4\,3],\\
  \end{cases}$$
\end{minipage}
\hspace{5mm}
\begin{minipage}{0.4\textwidth}
$$\delta^3 =
  \begin{cases}
     [4\,3\,2\,5\,1]&or\\
     [4\,5\,1\,2\,3]&or\\
     [2\,1\,3\,5\,4]&or\\
     [2\,5\,4\,3\,1].\\
  \end{cases}$$
\end{minipage}
\vspace{3mm}

As one can see, there are no a path of length two between $\delta^1$ and $\delta^3$. Thus, a sought $11$-cycle can not occur in this case. If $n=7$, then by~(\ref{delta1_22223}) and~(\ref{delta3_22223}) 
we have: \vspace{3mm}

\begin{minipage}{0.4\textwidth}
$$\delta^1 =
  \begin{cases}
     [4\,5\,6\,1\,2\,3\,7]&or\\
     [2\,3\,4\,1\,7\,6\,5]&or\\
     [4\,1\,2\,3\,7\,6\,5]&or\\
     [2\,1\,7\,6\,5\,4\,3],\\
  \end{cases}$$
\end{minipage}
\hspace{5mm}
\begin{minipage}{0.4\textwidth}
$$\delta^3 =
  \begin{cases}
     [4\,3\,2\,7\,6\,5\,1]&or\\
     [6\,5\,4\,7\,1\,2\,3]&or\\
     [4\,7\,6\,5\,1\,2\,3]&or\\
     [6\,7\,1\,2\,3\,4\,5].\\
  \end{cases}$$
\end{minipage}
\vspace{3mm}

\noindent Again, a sought $11$-cycle can not occur in this case, since there is no a path of length two between $\delta^1$ and $\delta^3$. By using similar arguments, one can check that for $n=9,11,13$ there is no a $11$-cycle in the graph. For any odd $n\geqslant 15$, there is a contradiction with an assumption that both $\delta^1$ and $\delta^3$ belong to the same copy. This complete the proof of the last case of Theorem~\ref{11-cycles}. \hfill $\square$

\end{proof}

\section{Proof of Theorem~\ref{girth-six-generating-sets}}\label{proof}

It is obvious that any cycle of the cubic Pancake graphs $P_n^i,\ i=1,\ldots,5,$ does belong to the Pancake graph $P_n, n\geqslant 4$, and should be described by one of the canonical formulas from Theorem~\ref{67-cycles} and Theorem~\ref{8-cycles}. Let us check which cycles appear in $P_n^i$ for each  $i\in\{1,\ldots,5\}.$ \\

\noindent {\bf Case (1.1)} Any cycle of $P_n^1, n\geqslant 4$, is formed by prefix--reversals from the set $BS_1=\{r_2,r_{n-1},r_n\}$. If $n=4$ then by Theorem~\ref{67-cycles} the prefix--reversals $r_2$ and $r_3$ give $6$-cycles of the form~(\ref{C6}). For $n\geqslant 4$, by~(\ref{C7}) there are no $7$--cycles in $P_n^1$, but there are $8$--cycles of the form~(\ref{C82}) when $n=4$ and there are $8$--cycles of the form~(\ref{C87}) for $n\geqslant 5$ if we put $k=n$, $i=2$, $j=n-1$. The canonical form in the last case is given by $(r_n\,r_2)^4$ for any $n\geqslant 5$. Hence, the formula~(\ref{BS123}) holds.

Using similar arguments, one can see that any cycle of $P_n^2, n\geqslant 4$, is presented by prefix--reversals from the set $BS_2=\{r_{n-2},r_{n-1},r_n\}$. If $n=4$ then obviously $P_4^2$ has $6$--cycles. If $n>4$ then by~(\ref{C7}) there are no $7$--cycles in $P_n^2$, however by Theorem~\ref{8-cycles} there are $8$--cycles of the canonical form~(\ref{C81}). Indeed, if we put $k=n$, $j=n-1$, $i=n-2$ in~(\ref{C81}) then we have the sequence $r_n\,r_{n-1}\,r_{n-2}\,r_{n-1}\,r_n\,r_{n-1}\,r_{n-2}\,r_{n-1}$. Thus, the formula~(\ref{BS123}) holds in this case.

The same arguments appear for the graph $P_n^3, n\geqslant 4$ is even, whose generating set contains prefix--reversals $r_3$, $r_{n-2}$ and $r_n$. It has $6$--cycles of the form~(\ref{C6}) and $8$--cycles of the form~(\ref{C88}) if $n=4$, but it does not have $7$--cycles for $n\geqslant 4$. Moreover, for any even $n\geqslant 6$ in $P_n^3$ there are no $6$--cycles, but there are $8$--cycles of the form~(\ref{C87}) if we put $k=n$, $i=3$,  $j=n-2$. The canonical form of $8$--cycles in this case is given by $(r_n\,r_3)^4$ for any even $n\geqslant 6$. Thus, the formula~(\ref{BS123}) holds. \\

\noindent {\bf Case (1.2)}  In the case of the graph $P_n^4, n\geqslant 5$ is odd, its generating elements $r_3$, $r_{n-1}$ and $r_n$ give $8$--cycles of the canonical form $(r_n\,r_{n-1}\,r_n\,r_3)^2$ if we put $k=n$, $j=2$ and $i=3$ in the form~(\ref{C87}). Obviously, there are no $6$--cycles in the graph since $r_2$ does not belong to the generating set for any $n\geqslant 5$. Hence, the formula~(\ref{BS4}) holds for any odd $n\geqslant 5$. \\

\noindent {\bf Case (1.3)} The generating set $BS_5=\{r_{n-3},r_{n-1},r_n\}$ of the graph $P_n^5$, where $n\geqslant 5$ is odd, gives the canonical form $(r_5\,r_4\,r_5\,r_2)^2$ of $8$--cycles if we put $k=5$, $j=2$ and $i=2$ in the form~(\ref{C87}). By Theorem~\ref{67-cycles} there are no $6$--cycles in the graph for any $n\geqslant 5$. By Theorem~\ref{8-cycles} and by the characterization of $9$-cycles in the Pancake graph~\cite[Theorem~4]{KM11} there are no $8$-- and $9$--cycles in $P_n^5$ for any $n\geqslant 7$. By Theorem~\ref{10-cycles} and Theorem~\ref{11-cycles} for any $n\geqslant 7$ there are no $10$--cycles and $11$--cycles in $P_n^5$, and the smallest cycle in $P_n^5$ is $12$--cycle of the canonical form $C_{12}=(r_n\,r_{n-1}\,r_n\,r_{n-1}\,r_{n-3}\,r_{n-1})^2$. This complete the proof of Theorem~\ref{girth-six-generating-sets}. \hfill $\square$

\section{Acknowledgements}

The first author is supported by the project No.~FWNF-2022-0017 (the state contract of the Sobolev Institute of Mathematics). The second author was partially supported by Mathematical Center in Akademgorodok under agreement No. 075-15-2019-1613 with the Ministry of Science and Higher Education of the Russian Federation.

\end{document}